\documentclass[reqno,12pt]{amsart}
\usepackage[colorlinks=true, linkcolor=blue, citecolor=blue]{hyperref}
\usepackage{amssymb}
\usepackage{amsmath, graphicx, rotating}
\usepackage{color}
\usepackage{soul}
\usepackage[dvipsnames]{xcolor}

\usepackage{amsmath,amssymb,amsfonts,amsthm,amscd,amsxtra}
\usepackage{mathrsfs,eucal,dsfont}
\allowdisplaybreaks

\usepackage{ifthen}
\usepackage{xkeyval}
\usepackage{todonotes}
\usepackage[T1]{fontenc}
\usepackage{lmodern}
\usepackage[english]{babel}

\usepackage{ upgreek }
\usepackage{stmaryrd}
\SetSymbolFont{stmry}{bold}{U}{stmry}{m}{n}
\usepackage{amsthm}
\usepackage{float}

\usepackage{ bbm }
\usepackage{ stmaryrd }
\usepackage{ mathrsfs }
\usepackage{ frcursive }
\usepackage{ comment }

\usepackage{pgf, tikz}
\usetikzlibrary{shapes}
\usepackage{varioref}
\usepackage{enumitem}

\definecolor{rouge}{rgb}{0.7,0.00,0.00}
\definecolor{vert}{rgb}{0.00,0.5,0.00}
\definecolor{bleu}{rgb}{0.00,0.00,0.8}
\usepackage[margin=1in]{geometry}
\newtheorem{theorem}{Theorem}[section]
\newtheorem*{theorem*}{Theorem}
\newtheorem{lemma}[theorem]{Lemma}

\newtheorem{proposition}[theorem]{Proposition}

\labelformat{hypothesis}{\textbf{M\kern-0.1mm#1}}

\newtheorem{remark}[theorem]{Remark}
\newtheorem{condition}{Condition}

\theoremstyle{definition}
\numberwithin{equation}{section}

\newcommand{\I}{\mathds 1}

\labelformat{conditionA}{\textbf{A\kern-0.1mm#1}}

\def\bf#1{\mathbf{#1}}
\def\scr#1{\mathscr{#1}}

\def\tt#1{\tilde{#1}}

\def\d{{\rm d}}

\def\a{{\rm ~and~}}

\def\8{{\infty}}

\def\R{\mathbb R}
\def\F{\mathbb F}
\def\P{\mathbb{P}}
\def\E{\mathbb{E}}

\def\W{{\mathbb W}}

\def\<{\langle}
\def\>{\rangle}
\def\wt{\widetilde}

\def\be{{\beta}}

\def\Om{{\Omega}}

\def\al{{\alpha}}

\def\be{{\beta}}

\def\Ga{{\Gamma}}
\def\ga{{\gamma}}

\def\si{{\sigma}}

\def\vv{{\varepsilon}}
\def\la{\langle}
\def\ra{\rangle}

\def\nn{\nabla}

\DeclareMathOperator{\e}{e}

\begin{document}
\allowdisplaybreaks
\title[Well-posedness of L\'evy-driven McKean-Vlasov SDEs] {A note on   L\'evy-driven McKean-Vlasov SDEs under monotonicity}

\author{
Jianhai Bao \qquad Yao Liu \qquad
Jian Wang}
\date{}
\thanks{\emph{J.\ Bao:} Center for Applied Mathematics, Tianjin University, 300072  Tianjin, P.R. China. \url{jianhaibao@tju.edu.cn}}

\thanks{\emph{Y. \ Liu:} School  of Mathematics and Statistics, Fujian Normal University, 350007 Fuzhou, P.R. China. \url{liuyaomath@163.com}}

\thanks{\emph{J.\ Wang:}
School  of Mathematics and Statistics \& Key Laboratory of Analytical Mathematics and Applications (Ministry of Education) \& Fujian Provincial Key Laboratory
of Statistics and Artificial Intelligence, Fujian Normal University, 350007 Fuzhou, P.R. China. \url{jianwang@fjnu.edu.cn}}

\date{}

\maketitle

\begin{abstract}
In this note, under a weak  monotonicity and a weak coercivity,
we address strong well-posedness of McKean-Vlasov stochastic differential equations (SDEs) driven by L\'{e}vy jump processes, where the coefficients are Lipschitz continuous (with respect to the measure variable) under the $L^\beta$-Wasserstein distance for $\beta\in[1,2].$ Moreover, the issue on the weak propagation of chaos (i.e., convergence in distribution via the convergence of the empirical measure) and the strong propagation of chaos (i.e., at the level paths by coupling)
is explored simultaneously.  To treat the strong well-posedness of McKean-Vlasov SDEs we are interested in,  we  investigate   strong well-posedness of classical time-inhomogeneous SDEs with jumps under a local weak monotonicity and a global weak coercivity. Such a result is of independent interest, and, most importantly, can provide  an available  reference on strong well-posedness of L\'{e}vy-driven SDEs under the monotone condition, which nevertheless is missing  for a long time. Based on the  theory derived, along with the interlacing technique and the Banach fixed point theorem, the strong well-posedness of McKean-Vlasov SDEs driven by L\'{e}vy jump processes can be established. Additionally, as a potential extension, strong well-posedness and conditional  propagation of chaos are treated for L\'{e}vy-driven McKean-Vlasov SDEs with common noise
under a  weak monotonicity.

\medskip

\medskip

\end{abstract}
\allowdisplaybreaks

\section{Introduction and main results} \label{Sec Main Results}

\subsection{Background}
The treatment of strong/weak well-posedness is a  starter to explore  qualitative/quantitative studies of stochastic differential equations (SDEs for short) under consideration.
In the past few decades,  strong well-posedness of SDEs with Brownian motion noises has been investigated under various scenarios; see, for instance,
\cite{IW,Mao} for the Lipschitz continuity and the linear growth; \cite[Theorem 3.4]{Mao} concerning the local Lipschitz condition plus the linear growth;
\cite[Theorem 3.5]{Mao} regarding the local Lipschitz continuity along with the Lyapunov condition, and \cite[Theorem 3.1.1]{PR} with regard to the  local weak monotonicity besides the global weak coercivity. Moreover, under the local $L^q(L^p)$-condition,
strong well-posedness of singular SDEs has  also advanced greatly via  Zvonkin's transformation; see e.g. \cite{XXZZ,XZ}  and references therein.

At the same time, there is a considerable amount of  literature concerned with  strong well-posedness of SDEs driven by L\'evy jump process.
As we know, under the standard assumption that coefficients are globally Lipschitz and of linear growth,  SDEs with pure jumps  are strongly well-posed; see, for example,  \cite{AD,IW}. In  case that drifts and Brownian diffusions   are locally Lipschitz and the jump coefficient is globally Lipschitz, in addition to  a weak coercivity, strong well-posedness of  SDEs with jumps was explored in \cite{ABW}. With contrast to SDEs with Brownian motion noises, strong well-posedness of SDEs driven by pure jump processes is rare under the  weak monotonicity (which is also termed as the one-sided Lipschitz condition) and the weak coercivity. As stated in \cite{MM},   many authors quote strong well-posedness of SDEs with jumps under
the one-sided Lipschitz condition by {\it claiming that it is totally  well-known nevertheless  without providing any reference
or referring to references which do not contain it at all}. This phenomenon is further stressed in \cite{MSSZ} as follows: ``However, we could not find a reference in
the literature that covers our setting completely.'' Based on the point of view above,  via a truncation approach, \cite{MM} addressed  strong well-posedness of SDEs driven by Brownian motions and  compensated Poisson random measures, where
  a  local  weak monotonicity and a global  weak coercivity were imposed. Unsatisfactorily, due to the limitation of the method adopted,  the local  weak monotonicity and the global  weak coercivity put in \cite{MM} cannot go back to the classical one
   (see e.g. \cite[(3.1.3) and (3.1.4)]{PR} for more details) when  the pure jump term involved vanishes. Additionally,
we would like to mention \cite{GK}  for a much more general setup, where the driven noise is a square-integrable  semimartingale.

In recent years, there are great progresses  as well on strong well-posedness of McKean-Vlasov SDEs driven by Brownian motions;   see e.g.     monographs \cite{CDa,MS,WR}. We also would like to mention that  \cite{CDS} explored  strong well-posedness of McKean-Vlasov SDEs, which allow   drifts and diffusions to be of super-linear growth in measure and state  variables. Meanwhile, strong well-posedness of regular McKean-Vlasov SDEs with jumps
has also attracted a lot of interest; see e.g. \cite{DH,EX,LMW,MSSZ}.  In detail,  \cite{DH,LMW} is concerned with  the additive noise and the drift involved in \cite{EX} is of linear growth with respect to the state  variable. Furthermore,
strong well-posedness of McKean-Vlasov SDEs with singular interaction kernels and symmetric $\alpha$-stable noises has been tackled  in \cite{DHb,DHc} and \cite{FKM,HRW} by the aid of
the (two-step) fixed point theorem    and   the non-linear martingale problem, respectively.  Additionally, via a   Fourier-based Picard-iteration approach, \cite{AP} considered strong well-posedness of a class of
McKean-Vlasov SDEs with L\'{e}vy jumps, where the underlying drift coefficient is affine in the state variable.  When the coefficients are Lipschitz continuous with respect to the measure variable under the $L^2$-Wasserstein distance and
non-globally Lipschitz continuous with respect to the state variable,
 strong well-posedness of  L\'{e}vy-driven McKean-Vlasov SDEs
   has been treated in \cite{MSSZ,TKLN}. In   case that the associated coefficients are $L^\beta$-Wasserstein continuous (for $1\le \beta\le 2$) as far as the measure variable  is concerned, and Lipschitz continuous in the state variable,
   the paper  \cite{CT}  probed into   well-posedness of L\'{e}vy-driven McKean-Vlasov SDEs.

Inspired by the aforementioned literature, in this note we aim to investigate the strong well-posedness of a class of L\'{e}vy-driven McKean-Vlasov SDEs under a weak monotonicity and a weak coercivity, which will weaken the associated conditions and
 improve  the corresponding results in e.g. \cite{AP,CT,DH,MSSZ,TKLN} in various  aspects.

\subsection{Well-posedness of  McKean-Vlasov SDEs}
More precisely,  in this note we focus on  the following   McKean-Vlasov SDE on $\R^d$:
\begin{equation} \label{MV}
\d X_t=b(X_t,\scr{L}_{X_t})\,\d t+\int_U f(X_{t-},\scr{L}_{X_t},z)\,\wt N(\d t,\d z)+\int_V g(X_{t-},\scr{L}_{X_t},z)\,N(\d t,\d z).
\end{equation}
Herein, $\scr{L}_{X_t}$ stands for the law of $X_t$; $X_{t-}:=\lim_{s\uparrow t}X(s)$;
$
b:\R^{d}\times\scr{P}(\R^{d})\to\R^{d}$, and
$f,g:\R^{d}\times\scr{P}(\R^{d})\times \R^d\to\R^{d}
$
are measurable maps, where $\scr{P}(\R^d)$ means the family of probability measures on $\R^d$;  $U, V\subset \R^d_0:=\R^d\backslash \{{\bf0}\}$ so that $U\cap V=\emptyset$;
 $N$ is a Poisson random measure, carried on a complete probability space $(\Omega,\mathscr F,\P)$, with the intensity measure $\d t\times \nu(\d z)$ for a $\sigma$-finite measure $\nu(\d z)$, and $\wt N(\d t,\d z):=N(\d t,\d z)-
 \d t\nu(\d z)$ represents the associated compensated Poisson measure. Furthermore, we shall assume that
 for some $\beta\in [1,2]$  and any
 fixed  $x\in \R^d$ and $\mu\in \scr{P}(\R^d)$,
 \begin{equation}\label{e:welldefined}\nu(|f(x,\mu,\cdot)|^2\I_U(\cdot))+\nu((1
 \vee |\cdot|^\beta
 \vee|g(x,\mu,\cdot)|^\beta)\I_V(\cdot))<\8,
 \end{equation}
 where $\nu(f):=\int_{\R^d}f(x)\nu(\d x)$  for a $\nu$-integrable function $f:\R^d\to\R$.
 For  $p>0,$ denote by $\W_p$ the
 $L^p$-Wasserstein distance:
 $$
\W_p(\mu_1,\mu_2):=\inf_{\pi\in\mathscr{C}(\mu_1,\mu_2)}\left(\int_{\R^d\times\R^d}|x-y|^p\,\pi(\d x,\d y)\right)^{\frac{1}{1\vee p}},\quad \mu_1,\mu_2\in\mathscr P_p(\R^d),
$$
where $ \mathscr P_p(\R^d):=\{\mu\in\mathscr P(\R^d): \mu(|\cdot|^p)<\8\}$ and $\mathscr C(\mu_1,\mu_2)$ is the set of  couplings for  $\mu_1$ and $\mu_2$.

To guarantee the well-posedness of the SDE \eqref{MV}, the following assumptions are in force.

\begin{enumerate}\it
\item [{\rm(${\bf A}_1$)}]
for  fixed $\mu\in\scr{P}_\be(\R^d)$ and $z\in\R^d$,   $\R^d \ni  x \mapsto b(x,\mu)$ and $\R^d\ni x\mapsto f(x,\mu,z)$  are continuous and locally bounded, and there exists a constant $L_1>0$ such that
for any   $x,y,z\in\R^d$,
 and $\mu_1,\mu_2\in\scr{P}_\be(\R^d)$,
\begin{equation}\label{EQ}
\begin{split}
& 2\langle b(x,\mu_1)-b(y,\mu_2),x-y\rangle+  \nu(|f(x,\mu_1,\cdot)-f(y,\mu_2,\cdot)|^2\I_U(\cdot)) \\
&\le L_1(|x-y|^2+\W_\be(\mu_1,\mu_2)^2),
\end{split}
\end{equation}
and
\begin{align}\label{EQ*}
|g(x,\mu_1,z)-g(y,\mu_2,z)| \le L_1(1+|z|)(|x-y| +\W_\be(\mu_1,\mu_2) );
\end{align}

  \item [{\rm(${\bf A}_2$)}] there exists a constant $L_2>0$ such that for any $x\in\R^d$ and  $\mu\in\scr{P}_\be(\R^d)$,
\begin{equation*}
2\<x,b(x,\mu)\>
 +
 \nu(|f(x,\mu,\cdot)|^2\I_U(\cdot))
 \le L_2 (1+|x|^2+\mu(|\cdot|^\be)^{\frac{2}{\be}});
\end{equation*}

  \item [{\rm(${\bf A}_3$)}]for any $T,R>0$ and $\mu\in C([0,T];\mathscr P_\beta(\R^d))$,
\begin{align*}
\int_0^T \Big(\sup_{\{|x|\le R\}}|b(x,\mu_t)| +\int_U\sup_{\{|x|\le R\}}  |f (x,\mu_t,z)|^2\nu(\d z) \Big)\,\d t<\8.
\end{align*}
\end{enumerate}

The first main result in this paper is stated as follows.

\begin{theorem}\label{theorem2}
Assume that Assumptions $({\bf A}_1)$-$({\bf A}_{3})$   hold, and suppose further  $X_0\in L^\beta(\Omega\to\R^d,\mathscr F_0,\P)$.
Then, the McKean-Vlasov  SDE \eqref{MV} admits a unique strong solution $(X_t)_{t\ge0}$ satisfying that,
for any fixed $T>0$, there exists a constant $C_T>0$ such that
\begin{align}\label{B0}
\E|X_t|^\beta \le C_T(1+\E|X_0|^\be),\quad 0\le t\le T.
\end{align}
In addition, if Assumption
 $({\bf A}_2)$ is replaced by the following stronger  one: for some    $L_3>0,$
 \begin{equation}\label{A21}
 \<x,b(x,\mu)\> \vee
 \nu( |f(x,\mu,\cdot)|^2\I_U(\cdot))
 \le L_3 \big(1+|x|^2+\mu(|\cdot|^\be)^{\frac{2}{\beta}}\big ),
\end{equation}
then, for any $T>0$, there exists a constant $C_T'>0$ such that
\begin{align}\label{EQ-}
\E\Big(\sup_{0\le t\le T}  |X_t|^\be\Big)   \le C_T'(1+\E|X_0|^\be).
\end{align}
\end{theorem}

Below, we make some comments on Theorem \ref{theorem2} and assumptions  mentioned above.
\begin{remark}\label{EX}\rm
\begin{enumerate}
\item[(i)]  \eqref{EQ} shows that $b,f$  satisfy the so-called one-sided Lipschitz condition  so they
are allowed to be non-globally Lipschitz with respect to the state variable as the following example reveals.
For $x,z\in\R^d$ and $\mu\in\mathscr P_\beta(\R^d)$, let
\begin{align*}
 b(x,\mu)&=C_1 x-C_2 x|x|^2+\mu(|h(x-\cdot)|^\beta)^{\frac{1}{\beta}} {\bf 1},\\
 f(x,\mu,z)&= C_3z (1+C_4|x|^2+\mu(|h(x-\cdot)|^\beta)^{\frac{1}{\beta}} ),\\
 g(x,\mu,z)&=({\bf 1}+ z )(1+|x|+\mu(|h(x-\cdot)|^\beta)^{\frac{1}{\beta}}  ),
\end{align*}
where $C_1,C_2,C_3, C_4 >0 $, ${\bf1}:=(1,\cdots,1)^\top\in\R^d$, and $h:\R^d\to\R^d$ is Lipschitz. Then, Assumptions $({\bf A}_1)$, $({\bf A}_2)$ and $({\bf A}_3)$
are valid respectively provided  $\nu(|\cdot|^2\I_U(\cdot))<\8$
and
$C_2>12C_3^2C_4^2\nu(|\cdot|^2\I_U(\cdot))$.

\item[(ii)]For the case $\beta\in(0,1)$, there is no uniqueness of the McKean-Vlasov SDE \eqref{MV}
as  \cite[Remark 2]{CT} shown.
 See also Remark \ref{Add:re1}(i) below for additional comments.
  So, in this work we focus only on the setting $\beta\in[1,2].$ As for
 the case $U=
 \{z\in \R^d: 0<|z|\le 1\}$, $V=
 \{z\in \R^d: |z|>1\}$, $f(x,\mu,z)=\si(x,\mu)z$, and $g(x,\mu,z)=\si(x,\mu)z$ with $\si:\R^d\to\R^d\otimes\R^d$,
the strong well-posedness of \eqref{MV} was addressed in \cite[Theorem 1]{CT} when $b,\si$ are Lipschitz and $L^\beta$-Wasserstein Lipschitz with respect to the spatial variable  and the measure variable, respectively. Whereas, in Theorem \ref{theorem2},  $b$ and $f$ might  be non-globally Lipschitz with respect to the state variables as the previous example  demonstrates. In addition, \cite{DH} addressed the strong well-posedness of the McKean-Vlasov SDE \eqref{MV} with additive noise under the following condition: there is some $L_0>0$ such that for all $x,y\in\R^d$ and
$\mu_1,\mu_2\in\mathscr P_\beta(\R^d)$
\begin{align*}
\<x-y,b(x,\mu_1)-b(y,\mu_2)\>\le L_0(|x-y|+\mathbb W_\beta(\mu_1,\mu_2))|x-y| .
\end{align*}
Apparently, the preceding condition is rigorous than the one imposed in \eqref{EQ}.

\item[(iii)]
In addition to the fixed point theorem used in the proof of Theorem \ref{theorem2},
Yamada-Watanabe's principle is
another approach that is applied widely  to prove  strong well-posedness, e.g., see  \cite{HHL,HX,HW} concerning McKean-Vlasov diffusions and  \cite{DH,HF} for McKean-Vlasov SDEs with additive L\'evy noise. In particular, under the following local  Lipschitz continuity:  for some   $L_1>0$, and any  $x,y\in\R^d,\mu_1,\mu_2\in\mathscr P_1(\R^d),$
\begin{align*}
&|b(x,\mu_1)-b(y,\mu_2)|+\nu(|f(x,\mu_1,z)-f(y,\mu_2,z)|\I_{U}(|\cdot|))\\
&\le L_1(1+|x|+|y|+\mu_1(|\cdot|)+\mu_2(|\cdot|))(|x-y| +\W_1(\mu_1,\mu_2) ),
\end{align*}
 strong well-posedness of a kind of  McKean-Vlasov equations with jumps   was investigated in \cite{EX} by
Yamada-Watanabe's principle.
On the other hand, the strong existence
is implied by weak existence of McKean-Vlasov SDEs and  strong well-posedness
of the corresponding decoupled SDEs.
In literature,  the statement above is called the modified Yamada-Watanabe principle \cite[Lemma 3.4]{HX}.
In general,  some kind of growth condition needs to be  imposed to verify the tightness of the sequence of Euler-type approximation equations
in order to prove the weak existence. For example,  the growth condition: for some $L_2>0$, and any $x\in\R^d,$ $\mu\in\mathscr P_1(\R^d)$,
$$|b(x,\mu)|+|f(x,\mu,z)|
 \le L_2 (1+|z|)(1+|x|+\mu(|\cdot|)) $$
was set in  \cite{EX}.  Clearly, such condition is stronger than Assumption (${\bf A}_2$) in our paper. See Remark \ref{R:2.3} below for comments on the setting of classical SDEs with jumps.
\end{enumerate}
\end{remark}

\subsection{Propagation of chaos}
The McKean-Vlasov SDE \eqref{MV} arises naturally in the framework of the limit for the mean-filed interacting particle system
 in the form:
\begin{equation}\label{MF}
\begin{cases}
\d X^{i,n}_t=b(X^{i,n}_t,\bar{\mu}^n_t)\,\d t+\displaystyle\int_U f(X^{i,n}_{t-},\bar{\mu}^n_t,z)\,\wt N^i(\d t,\d z)\\
\quad\quad\quad\quad +\displaystyle\int_V g(X^{i,n}_{t-},\bar{\mu}^n_t,z)\,N^i(\d t,\d z),\\
X^{i,n}_0=X_0^i, \qquad i=1,2,\cdots,n,
\end{cases}
\end{equation}
where $\bar{\mu}^n_t:=\frac{1}{n}\sum_{i=1}^n\delta_{X^{i,n}_t}$, and $\{N^i(\d t, \d z)\}_{1\le i\le n}$ are independent Poisson measures with the intensity measure $\d t\times \nu(\d z)$.
The link between \eqref{MV} and \eqref{MF} lies in that the dynamics of the particle system \eqref{MF} are expected to be described by \eqref{MV} when the number of particles $n$ goes to infinity.
This property is the so-called propagation of chaos, which was originally studied by Kac \cite{KM} for the Boltzmann equation
and was further developed by Sznitman \cite{SA}.
The propagation of chaos can be interpreted in the weak sense (i.e., in the distribution through the convergence of the empirical measure $\bar{\mu}^n_t$) and  in the strong sense (i.e., from the point of view of paths via  coupling); see \cite{CT,ChD1,ChD2} and references within.

In this subsection,
our purpose is to prove quantitative propagation of chaos both in the weak sense and the strong sense, respectively, concerning the mean-field interacting particle system \eqref{MF}
with $f(x,z)=f(x,\mu,z)$ (i.e., $f$ is unrelated to the measure variable).
 For this,
 we need to replace \eqref{EQ} in Assumption (${\bf A}_1$) by the following stronger version:
\begin{itemize} \it
  \item [{\rm(${\bf A}_1'$)}]
  assume that $\beta\in (1,2]${\rm;}
for  fixed $\mu\in\scr{P}_\be(\R^d)$ and $z\in\R^d$,   $\R^d \ni  x \mapsto b(x,\mu)$ and $\R^d\ni x\mapsto f(x, z)$  are continuous and locally bounded, and for some
$p\in[1,\be)$, there exists a constant $L_4>0$ such that
for any   $x,y,z\in\R^d$,
 and $\mu_1,\mu_2\in\scr{P}_p(\R^d)$,
\begin{equation}\label{EQ4}
\begin{split}
& 2\langle b(x,\mu_1)-b(y,\mu_2),x-y\rangle+  \nu(|f(x,\cdot)-f(y, \cdot)|^2\I_U(\cdot)) \\
&\le L_4(|x-y|+\W_p(\mu_1,\mu_2))|x-y|,
\end{split}
\end{equation}
and
\begin{align}\label{EQ5}
|g(x,\mu_1,z)-g(y,\mu_2,z)| \le L_4(1+|z|)(|x-y| +\W_p(\mu_1,\mu_2) ).
\end{align}
\end{itemize}

Under (${\bf A}_1'$), (${\bf A}_2$) and (${\bf A}_3$), besides \eqref{e:welldefined},
 SDEs  \eqref{MV} and \eqref{MF}  have   unique strong solutions for  $X_0\in L^\beta(\Omega\to\R^d,\mathscr F_0,\P)$ by taking   Theorem \ref{theorem2} into consideration.

 Let $\{(X_t^i)_{t\ge0}\}_{1\le i\le n}$ be $n$-independent versions of the unique solution to the SDE \eqref{MV}
 with $f(x,z)=f(x,\mu,z)$.
 In particular, each $(X_t^i)_{t\ge0}$, $1\le i\le n$, shares the same distribution.  The following theorem provides quantitative characterizations of strong/weak  propagation of chaos in finite time.

\begin{theorem}\label{POC}
Assume that Assumptions $({\bf A}_1')$, $({\bf A}_{2})$ and $({\bf A}_{3})$ hold, and suppose further  $X^i_0\in L^\beta(\Omega\to\R^d,\mathscr F_0,\P)$ for any $1\le i\le n$.
Let $(\mu_t)_{t\ge0}$ be the common distribution of $(X_t^i)_{t\ge0}$ for all $1\le i\le n$, and $\bar{\mu}^n_t=\frac{1}{n}\sum_{i=1}^n\delta_{X^{i,n}_t}$, where $\{(X_t^{i,n})_{t\ge0}\}_{1\le i\le n}$ is the solution to the mean-filed interacting particle system
\eqref{MF}  with   $f(x,z)=f(x,\mu,z)$.
Then, for any fixed $T>0$, there is a constant $ C_T>0$ such that
\begin{equation}\label{wPOC}
\begin{split}
 \E\W_p^p(\bar{\mu}^n_t,\mu_t)
\le C_T\phi_{p,\beta,d}(n),\quad t\in[0,T],
\end{split}
\end{equation}
where the quantity $  \phi_{p,\beta,d}(n)$
is defined as below:
\begin{equation}\label{EE*}
\phi_{p,\beta,d}(n):=
\begin{cases}
n^{-(1-\frac{p}{\be})}, & d=1,2;  d=3 \a \frac{3}{2}\le p<\be; \\
                            &d=3, 1\le p<\frac{3}{2} \a \be<\frac{3p}{3-p}; d\ge4 \a \be<\frac{dp}{d-p}; \\
    n^{-\frac{p}{d}}, & d=3, 1\le p<\frac{3}{2} \a \be\ge\frac{3p}{3-p}; d\ge4 \a \be\ge\frac{dp}{d-p}.
\end{cases}
\end{equation}
Furthermore, for  any $0\le q_1<q_2<1,$
there exists a constant $\hat C_T>0$ such that
\begin{equation}\label{sPOC}
\begin{split}
\E\Big(\sup_{0\le t\le T}|X^{i,n}_t-X_t^i|^{pq_1}\Big)\le \frac{q_2}{q_2-q_1}\big(\hat C_T
\phi_{p,\beta,n}(n)\big)^{q_1}.
\end{split}
\end{equation}
\end{theorem}

To proceed, concerning Theorem \ref{POC},
we make a remark on the structure of $f$ and the prerequisite \eqref{EQ4}.
\begin{remark}\rm
Recall  that the well-posedness of \eqref{MV} is explored under $({\bf A}_1)$ via the interlacing trick. So, it is quite natural to investigate the issue on propagation of chaos under $({\bf A}_1)$. Whilst, the empirical measure involved in \eqref{MF} is random so
the above technique
does not work any more.
In turn, we  reinforce Assumption \eqref{EQ}
as \eqref{EQ4}.
 Once $f$ is dependent on the measure variable $($as given in \eqref{MF}$)$, the assumption \eqref{EQ4} can be formulated  as below: for any   $x,y \in\R^d$,
 and $\mu_1,\mu_2\in\scr{P}_p(\R^d)$,
\begin{equation*}\label{EQ4*}
\begin{split}
& 2\langle b(x,\mu_1)-b(y,\mu_2),x-y\rangle+  \nu(|f(x,\mu_1,\cdot)-f(y, \mu_2,\cdot)|^2\I_U(\cdot)) \\
&\le L_4(|x-y|+\W_p(\mu_1,\mu_2))|x-y|.
\end{split}
\end{equation*}
In the preceding inequality, in case of $b\equiv{\bf0}$, one has $\nu(|f(x,\mu_1,\cdot)-f(x, \mu_2,\cdot)|^2\I_U(\cdot))=0$
for arbitrary  $\mu_1,\mu_2\in\scr{P}_p(\R^d)$. Accordingly, we can conclude that $f$ is irrelevant  to the measure variable
(at least when $b\equiv{\bf0}$).
Let $b$ be defined as in Remark \ref{EX}${\rm(i)}$ and $f(x,z)=
C_3z(1+C_4|x|^2)$ for some $C_3, C_4>0$. For this case,
\eqref{EQ4} is valid
when the assumptions in Remark \ref{EX}${\rm(i)}$ are satisfied
 by examining the proof of Remark \ref{EX}${\rm(i)}$; see the end of Section \ref{sec3} for more details.
\end{remark}

\ \

The content of this paper is organized as follows. In Section \ref{proofs},  via a Picard iteration approach, we investigate strong well-posedness of classical time-inhomogeneous  SDEs with L\'{e}vy noises under a local weak monotonicity and a weak coercivity, which is quite interesting in it's own right.  Also,
by invoking the interlacing technique, a uniform moment estimate in a finite horizon is   established in Section \ref{proofs}.   Based on the theory derived in Section \ref{proofs}, along with the Banach fixed point theorem and the interlacing technique, the proof of Theorem \ref{theorem2} is complete in Section \ref{sec3}. In addition, the remaining part of Section \ref{sec3} is devoted to
the proof of Theorem \ref{POC}, which is concerned with the weak propagation of chaos and the associated strong version.   In the last section, we extend accordingly Theorems \ref{theorem2} and \ref{POC} to L\'{e}vy-driven McKean-Vlasov SDEs with common noise.

\section{Well-posedness of classical SDEs with L\'evy noises}\label{proofs}
The fixed point theorem is one of the powerful tools to investigate  well-posedness of McKean-Vlasov SDEs under variant settings.
For this purpose, the corresponding distribution-frozen SDE (which, in  literature,  is also named as a decoupled SDE ) need to be considered.
In other words, by the aid of the decoupled SDE (which definitely is a time-inhomogeneous SDE), along with the fixed point theorem,
the well-posedness of McKean-Vlasov SDEs can be  treated. Inspired by the aforementioned routine,
in this section, we focus on the following time-inhomogeneous SDE: for any $t\ge0,$
\begin{equation} \label{e}
\d X_t=b_t(X_t)\,\d t+\int_U f_t({ X_{t-}},z)
\,\wt N(\d t,\d z)+\int_V g_t(X_{t-},z)
\,N(\d t,\d z),
\end{equation}
where $b:[0,\8)\times\R^d\to\R^d,$ and $f,g:[0,\8)\times\R^d\times \R^d\to\R^d$ are jointly measurable;
the subsets $U,V$, and the random measures $N(\d t,\d z)$, $\wt N(\d t,\d z)$ are untouched  as those in \eqref{MV}.
 In this section, we assume that
 for all $t\ge0$ and $x\in \R^d$, $$\nu(|f_t(x,\cdot)|^2\I_U(\cdot))+\nu((|g_t(x,\cdot)|^\beta\vee 1)\I_V(\cdot))<\8,$$ where $\beta\in  (0,2]$. Herein,  we emphasize that the results to be derived in this section hold true for all $\beta\in (0,2]$ instead of  $\beta\in [1,2]$.
 In general,
the second stochastic integral and the third one on the right hand side of \eqref{e} are  concerned   with   small jumps and   big jumps, respectively.

\subsection{Main results}
Inspired by the diffusive setting (e.g. \cite[Theorem 3.1.1]{PR}),
to address the well-posedness of \eqref{e}, we impose the following local weak monotonicity, weak coercivity, and local integrability on the coefficients. In detail,  we shall assume that
\begin{enumerate}\it
\item[{\rm(${\bf H}_1$)}]  for each fixed $t\ge0$ and $z\in\R^d,$ $\R^d\ni x\mapsto b_t(x)$, $\R^d\ni x\mapsto f_t(x,z)$ and $\R^d\ni x\mapsto g_t(x,z)$ are continuous and locally bounded;
for any fixed $R>0$, there exists an increasing and  locally integrable   function $[0,\8)\ni t\mapsto K_t(R) $ such that
for any $x,y\in\R^d$ with $|x|\vee |y|\le R$ and  $t\ge0$,
\begin{align}\label{E5}
2\<x-y,b_t(x)-b_t(y)\>+ \nu( |f_t(x,\cdot)-f_t(y,\cdot)|^2\I_U(\cdot))\le K_t(R)|x-y|^2;
\end{align}

\item[{\rm(${\bf H}_2$)}]  there exists an increasing and  locally integrable function $\varphi :[0,\8)\to(0,\8)$ such that  for $x\in\R^d$
 any $t\ge0$,
\begin{align*}
2\<x,b_t(x)\>
 &+\nu( |f_t(x,\cdot)|^2\I_U(\cdot))\\
 &+ \be^{-1}2^{\frac{\be}{2}+1} (1+|x|^2)^{1-\frac{\be}{2}}  \nu( |g_t(x,\cdot)|^\be\I_V(\cdot)) \le \varphi(t)(1+|x|^2);
\end{align*}
\item[{\rm(${\bf H}_3$)}] for any $R,T>0,$
\begin{align*}
\int_0^T\Big(\sup_{\{|x|\le R\}}|b_t(x)|+\int_U\sup_{\{|x|\le R\}}  |f_t(x,z)|^2\nu(\d z)\Big)\,\d t<\8.
\end{align*}
\end{enumerate}

Under assumptions above,  \eqref{e} is strongly well-posed as the theorem below states.

\begin{theorem}\label{thm}
Assume that
$({\bf H}_{1})$-$({\bf H}_{3})$ hold true, and suppose further
 $X_0\in L^\beta(\Omega\to\R^d,\mathscr F_0,\P)$.
Then, \eqref{e} has a unique strong solution $(X_t)_{t\ge0}$ satisfying  that, for any $T>0,$ there exists a constant $C_T>0$ such that
\begin{align}\label{E12}
 \sup_{0\le t\le T}\E|X_t|^\beta \le C_T(1+\E|X_0|^\beta).
\end{align}

\end{theorem}

To achieve a uniform moment estimate in a finite horizon, we strengthen $({\bf H}_{2})$ as follows:
\begin{enumerate}\it
\item[{\rm(${\bf H}_2'$)}]there exists an increasing and   locally integrable function $\phi:[0,\8)\to(0,\8)$ such that for any $x\in\R^d$ and $t\ge0,$
\begin{align*}
\<x,b_t(x)\>\vee\nu(|f_t(x,\cdot)|^2\I_U(\cdot) )\vee\big(\nu(|g_t(x,\cdot)|^\beta\I_V(\cdot) )\big)^{\frac{2}{\beta}}\le \phi(t)(1+|x|^2).
\end{align*}
\end{enumerate}

With Assumption (${\bf H}_2'$) at hand, a stronger version of \eqref{E12}   can be obtained.

 \begin{theorem}\label{theorem1}
 Assume  that
 $({\bf H}_1)$, $({\bf H}_2')$ and $({\bf H}_{3})$  hold true, and suppose further
  $X_0\in L^\beta(\Omega\to\R^d,\mathscr F_0,\P)$. Then, for any fixed $T>0$, there is a constant $C'_T>0$ such that
 \begin{equation}\label{moment}
 \E\Big(\sup_{0\le t\le T} |X_t |^\be\Big)\le
 C'_T (1+\E|X_0|^\beta ).
 \end{equation}
 \end{theorem}

Before we move
 to the next subsection, we make the following remark.
\begin{remark}\label{R:2.3}\rm
In
literature, there are several ways to show strong well-posedness of SDEs under consideration. In particular, as long as  the weak existence and the pathwise uniqueness are available, the strong well-posedness can be derived by leveraging on the Yamada-Watanabe theorem; see \cite{JJ} for classical SDEs driven by semimartingales.
Concerning the aforementioned method, one needs to examine  tightness of the solution processes associated with the   approximated  SDEs.  To this end, in general,  some growth conditions  $($with respect to the state variable$)$ related to  coefficients are  imposed to show
the  equicontinuity in probability.
Nevertheless, by adopting the procedure in the present work, the growth condition on each coefficient can be neglected.
\end{remark}

\subsection{Well-posedness of SDEs with small jumps}
In this subsection, we adopt a two-step
strategy  to explore the well-posedness of \eqref{e}. Firstly,
 we establish   the well-posedness of
the SDE (without big jumps): for any $t>0,$
\begin{equation}\label{E1}
\d Y_t=b_t(Y_t)\,\d t+\displaystyle\int_U f_t(Y_t,z)\,\wt N(\d t,\d z),\quad Y_0=X_0.
\end{equation}
Afterwards, the well-posedness of \eqref{e} can be tackled  by   splicing together big jumps via
the so-called interlacing technique (see e.g. \cite[p.\ 112--113]{AD} and \cite[p.\ 244--246]{IW}).

\begin{proposition}\label{pro1}
Under Assumptions of Theorem $\ref{thm}$ with $g\equiv{\bf0}$, \eqref{E1} has a unique strong solution $(Y_t)_{t\ge0}$ satisfying that,
  for any $T>0,$ there exists a constant $C_T>0$ such that
\begin{align}\label{moment1}
\sup_{0\le t\le T}\E|Y_t|^\beta\le C_T (1+\E|Y_0|^\beta ).
\end{align}
\end{proposition}

To address  the well-posedness of \eqref{E1}, we appeal to the Picard iteration approach. So, in the sequel,
  we work with the following iterated SDE: for any $t>0$ and integer $n\ge1,$
\begin{align}\label{E2}
\d Y_t^{(n)}=b_t(Y_{t_n}^{(n)})\,\d t+\int_U f_t(Y_{t_n}^{(n)},z)\,\wt N(\d t,\d z),\quad  Y_0^{(n)}=Y_0,
\end{align}
where $t_n:=\lfloor tn\rfloor /n$ with $\lfloor\cdot\rfloor$ being the floor function.
Below, for the notation brevity, we set
\begin{align*}
 p_t^{(n)}:=Y_{t_n}^{(n)}-Y_t^{(n)}\quad \mbox{ and }  \quad p_{t-}^{(n)}:=Y_{t_n}^{(n)}-Y_{t-}^{(n)}.
\end{align*}
Additionally, for any $R>0,$ we define the stopping time $\tau_R^{(n)}$ by
\begin{align*}
\tau_R^{(n)}=\inf\{t\ge0:|Y_t^{(n)}|>R/2\}.
\end{align*}

To accomplish the proof of Proposition \ref{pro1}, we prepare several preliminary  lemmas, where  the following  one shows that,
 for fixed $R>0,$
$\I_{(0,\tau_R^{(n)}]}(t) |p_t^{(n)}|\to0$ in probability as $n\to\8$.

\begin{lemma}\label{lem1}
Under $({\bf H}_3)$, for any   $R,\vv>0$,
\begin{align}\label{E4}
\lim_{n\to\8}\P\big(|p_t^{(n)}|\ge\vv,0<t\le \tau_R^{(n)}\big)=0.
\end{align}
\end{lemma}

\begin{proof}
From \eqref{E2}, we obviously have for any $t\ge0,$
\begin{align*}
p_t^{(n)}=-\int_{t_n}^tb_s(Y_{s_n}^{(n)})\,\d s -\int_{t_n}^t\int_U f_s(Y_{s_n}^{(n)},z)\,\wt N(\d s,\d z).
\end{align*}
This thus  implies that for any  $\vv>0,$
\begin{align*}
\P\big(|p_t^{(n)}|\ge\vv,0<t\le \tau_R^{(n)}\big)
&\le \P\bigg( \int_0^{t\wedge\tau_R^{(n)}}|b_s(Y_{s_n}^{(n)})|\I_{[t_n,t]}(s)\,\d s \ge\frac{\vv}{2}\bigg)\\
&\quad+\P\bigg( \bigg|\int_0^{t\wedge\tau_R^{(n)}}\int_U f_s(Y_{s_n}^{(n)},z)\I_{[t_n,t]}(s) \, \wt N(\d s,\d z)\bigg|\ge\frac{\vv}{2}\bigg)\\
&=:\Gamma_1(t,n,R,\vv)+\Gamma_2(t,n,R,\vv)
.
\end{align*}
On the one hand, via   Chebyshev's inequality and the definition of $\tau_R^{(n)}$, we deduce that
\begin{equation}\label{E7}
\begin{split}
\Gamma_1(t,n,R,\vv)
\le&\frac{2}{\vv}\E\bigg(\int_0^{t\wedge\tau_R^{(n)}}|b_s(Y_{s_n}^{(n)})|\I_{[t_n,t]}(s)\,\d s \bigg)
\le \frac{2}{\vv}\int_{t_n}^t\sup_{\{|x|\le R/2\}}|b_s(x)|\,\d s.
\end{split}
\end{equation}
On the other hand, using Chebyshev's inequality once more followed by It\^o's isometry and the notion of $\tau_R^{(n)}$ yields that
\begin{equation}\label{E8}
\begin{split}
\Gamma_2(t,n,R,\vv)&\le \frac{4}{\vv^2}\E\bigg|\int_0^{t\wedge\tau_R^{(n)}}\int_U f_s(Y_{s_n}^{(n)},z) \I_{[t_n,t]}(s) \,\wt N(\d s,\d z)\bigg|^2\\
&=\frac{4}{\vv^2}\E\bigg(\int_0^{t\wedge\tau_R^{(n)}}\int_U |f_s(Y_{s_n}^{(n)},z)|^2\I_{[t_n,t]}(s) \,\nu(\d z) \d s\bigg)\\
&\le  \frac{4}{\vv^2}\int_{t_n}^t \int_U \sup_{\{|x|\le R/2\}}|f_s(x,z)|^2 \,\nu(\d z) \d s.
\end{split}
\end{equation}
Therefore,   \eqref{E4} is valid  by
combining \eqref{E7} with \eqref{E8}, and taking $({\bf H}_3)$  into account.
\end{proof}

Roughly speaking,
the next lemma   indicates that the life time of $(Y_t^{(n)})_{t\ge0}$ goes to infinity.
\begin{lemma}\label{lem2}
Assume that  $({\bf H}_{2})$ with $g\equiv{\bf0}$
 and $({\bf H}_{3})$ holds true, and suppose    $Y_0\in L^\beta(\Omega\to\R^d,\mathscr F_0,\P)$.  Then,  for any fixed $T>0,$
\begin{align}\label{E9}
\lim_{R\to\8}\lim_{n\to\8}\P(\tau_R^{(n)}\le T)=0.
\end{align}
\end{lemma}

\begin{proof}
Apparently, we have
\begin{align}\label{E0}
\big\{\tau_R^{(n)}\le T\big\}=\bigg\{\tau_R^{(n)}\le T,\sup_{t\in[0,\tau_R^{(n)}]
  }|Y_t^{(n)}|\ge R/4\bigg\}+\bigg\{\tau_R^{(n)}\le T,\sup_{t\in[0,\tau_R^{(n)}]
 }|Y_t^{(n)}|< R/4\bigg\}.
\end{align}
In terms of the definition of $\tau_R^{(n)}$, it is obvious that the second event on the right hand side of \eqref{E0} is empty. Hence, the following implication and equivalence
\begin{align*}
\big\{\tau_R^{(n)}\le T\big\}\subseteq\bigg\{\sup_{t\in[0,T\wedge\tau_R^{(n)}]}|Y_t^{(n) }|\ge R/4 \bigg\}=\bigg\{\big|Y_{\tau_R^{(n),*}}^{(n) }\big|\ge R/4\bigg\}
\end{align*}
are available,
where
$$
\tau_R^{(n),*}:=T\wedge\tau_R^{(n)}\wedge\inf\big\{t\ge0:|Y_t^{(n)}|\ge R/4\big\}.
$$
By means of Chebyshev's inequality, one has
\begin{align*}
\P\big(\tau_R^{(n)}\le T\big)\le \frac{1}{(1+(R/4)^2)^{{\beta}/{2}}}\E V_{\beta}\big(Y_{\tau_R^{(n),*}}^{(n)}\big),
\end{align*}
where
$V_{\beta}(x):=(1+|x|^2)^{{\beta}/{2}}$ for all $ x\in\R^d.$
 Next, we define a function below:
\begin{align}\label{E*}
 \Psi_t(z_1,z_2)=\big(\varphi(t)  |z_1|  + |b_t(z_2) | \big)|z_1|,\quad t\ge0, z_1,z_2\in\R^d.
 \end{align}
As long as there exists a constant $ \bar{C}_T>0$ such that
\begin{align}\label{E10}
\E V_{\beta}\big(Y_{\tau_R^{(n),*}}^{(n)} \big) \le \bar{C}_T\bigg(  \E V_{\beta}(Y_0)+ \E\Big(\int_0^{\tau_R^{(n),*}} \Psi_s\big(p_{s-}^{(n)},Y_{s_n}^{(n)}\big)  \,\d s\Big)\bigg)
\end{align}
and
\begin{align}\label{E11}
\lim_{n\to\8}\E\Big(\int_0^{\tau_R^{(n),*}} \Psi_s\big(p_{s-}^{(n)},Y_{s_n}^{(n)}\big) \,\d s\Big)=0,
\end{align}
the statement  \eqref{E9} is verifiable   by approaching firstly $n\to\8$ followed by sending $R\to\8.$ Based on the preceding analysis, it amounts to verifying respectively \eqref{E10} and \eqref{E11} for the establishment of \eqref{E9}.

Applying the It\^o formula, we deduce from  \eqref{E2} that
\begin{align*}
\d\Big(\e^{-\beta\int_0^t\varphi(s)\,\d s}V_{\beta}\big(Y_t^{(n)}\big)\Big)=&\e^{-\beta\int_0^t\varphi(s)\,\d s}\Big(-\beta\varphi(t)V_{\beta}\big(Y_t^{(n)}\big)+(\mathscr L_t^0 V_{\beta})\big(Y_t^{(n)},Y_{t_n}^{(n)}\big) \Big)\,\d t\\
&+\d M_t^{(n),\beta },
\end{align*}
where $(M_t^{(n),\beta })_{t\ge0}$ is a local martingale, and for $h\in C^2(\R^d)$,
  $x,y\in\R^d$ and $t\ge0,$
\begin{equation}\label{KL}
\begin{split}
(\mathscr L_t^0h)(x,y):=&\<\nn h(x), b_t(y)\>\\
&+\nu\big((h(x+f_t(y,\cdot))-h(x)-\<\nn h(x),f_t(y,\cdot)\>)\I_U(\cdot)\big).
\end{split}
\end{equation}
For $F_\beta(r):=(1+r)^{\frac{\beta}{2}}$ on $[0,\infty)$,
note readily that  $F_\beta'(r)=\frac{\beta}{2}(1+r)^{\frac{\beta}{2}-1}$ and $F_\beta''(r)<0$
for $\beta\in(0,2].$
Thus, an application of Taylor's expansion yields that  for any $x,y\in\R^d$ and $t\ge0,$
\begin{equation}\label{E3}
\begin{split}
(\mathscr L_t^0V_{\beta})(x,y)&\le  \frac{1}{2}\beta(1+|x|^2)^{\frac{\beta}{2}-1}\big(2\<x,b_t(y)\>+\nu( |f_t(y,\cdot)|^2\I_U) \big)\\
&=\frac{1}{2}\beta(1+|x|^2)^{\frac{\beta}{2}-1}\big(2\<y,b_t(y)\>+\nu( |f_t(y,\cdot)|^2\I_U) \big)\\
&\quad+\frac{1}{2}\beta(1+|x|^2)^{\frac{\beta}{2}-1}\<x-y,b_t(y)\>\\
&\le \frac{1}{2}\beta \varphi(t)(1+|x|^2)^{\frac{\beta}{2}-1}(1+|y|^2)+ |x-y|\cdot|b_t(y)|\\
&\le \beta \varphi(t)V_{\beta}(x)+\big(2 \varphi(t)|x-y|+|b_t(y)|\big)|x-y|,
\end{split}
\end{equation}
where in the second inequality we used $({\bf H}_{2 })$ with $g\equiv{\bf0}$ and $\beta\in(0,2]$, and in the last display we utilized $\beta\in(0,2]$ and
$|y|^2\le 2|x|^2+2|x-y|^2$. Subsequently, \eqref{E3}, besides $\varphi>0,$
 implies that
\begin{equation}\label{GH}
\begin{split}
\e^{-\beta\int_0^T\varphi(s)\,\d s}\E V_{\beta}\big(Y_{\tau_R^{(n),*}}^{(n)} \big)
&\le\E\Big(\e^{-\beta\int_0^{\tau_R^{(n),*}}\varphi(s)\,\d s} V_{\beta}\big(Y_{\tau_R^{(n),*}}^{(n)} \big)\Big)\\
&\le  \E V_{\beta}(Y_0)
+2\E\Big(\int_0^{\tau_R^{(n),*}}  \Psi_s\big(p_{s-}^{(n)},Y_{s_n}^{(n)}\big) \,\d s\Big),
\end{split}
\end{equation}
where $\Psi$ was defined in \eqref{E*}.
As a result, \eqref{E10} is attainable by making use of the locally integrable property of $\varphi$.

We proceed to verify \eqref{E11}. Due to the definition of $\tau_R^{(n)}$, we
infer  that for any $t\in(0,\tau_R^{(n),*}],$
\begin{align*}
\Psi_t\big(p_{t-}^{(n)},Y_{t_n}^{(n)}\big)
\le \Lambda(t,R)|p_{t-}^{(n)}|\le R\Lambda(t,R),
\end{align*}
where
$
\Lambda(t,R):= R
\varphi(t)
+\sup_{|x|\le R/2} |b_t(x)|.
$
Hence,  by virtue of $\tau_R^{(n),*}\le T\wedge\tau_R^{(n)}$,
it is easy to see that   for any integer $m\ge1$,
\begin{equation}\label{e:addcomment}\begin{split}
\int_0^{\tau_R^{(n),*}}\Psi_s\big(p_{s-}^{(n)},Y_{s_n}^{(n)}\big) \,\d s
&\le \int_0^{\tau_R^{(n),*}}\Lambda(s,R)|p_{s-}^{(n)}| \,\d s\\
&\le m\int_0^{\tau_R^{(n) }} |p_{s-}^{(n)}| \,\d s +  R \int_0^T\Lambda(s,R) \I_{\{\Lambda(s,R)> m\}} \,\d s.\end{split}
\end{equation}
Next, Lemma \ref{lem1}, Fubini's theorem as well as the dominated convergence theorem yield that
\begin{align*}
\lim_{n\to\8}\E\Big(\int_0^{\tau_R^{(n) }} |p_{s-}^{(n)}| \,\d s\Big)=0.
\end{align*}
Consequently,  \eqref{E11} is achievable  by
exploiting the locally integrable property of $t\mapsto \Lambda(t,R)$, thanks to $({\bf H}_{3 })$, and  sending $n\to\8$ followed by $m\to\8$.
\end{proof}

With Lemmas \ref{lem1} and \ref{lem2} at hand,
 we move forward  to show that
 $(Y_t^{(n)})_{t\ge0}$ is a Cauchy sequence in  the sense of uniform convergence in probability, which is stated precisely as follows.

\begin{lemma}\label{Lem3}
Assume that  $({\bf H}_1)$, $({\bf H}_{2})$ with $g\equiv{\bf0}$
 and $({\bf H}_{3})$ are satisfied, and suppose that   $Y_0\in L^\beta(\Omega\to\R^d,\mathscr F_0,\P)$.  Then, for any $T,\vv>0,$
\begin{align}\label{E13}
\lim_{n,m\to\8}\P\Big(\sup_{t\in[0,T]}|Y_t^{(n)}-Y_t^{(m)}|\ge\vv\Big)=0.
\end{align}
\end{lemma}

\begin{proof}
Below,  we   set $Q_t^{n,m}:=Y_t^{(n)}-Y_t^{(m)},t\ge0,$ for notation brevity.
It is obvious  that
\begin{align*}
\P\Big(\sup_{t\in[0,T]}\big|Q_t^{n,m}\big|\ge\vv\Big)&\le \P\Big(\sup_{t\in[0,\tau_R^{n,m}]}\big|Q_t^{n,m}\big|\ge\vv \Big) +\P\big(\tau_R^{(n)} < T\big)+\P\big(  \tau_R^{(m)}<T\big),
\end{align*}
where $\tau_R^{n,m}:=T\wedge \tau_R^{(n)}\wedge \tau_R^{(m)}$.
Note that the following equivalence
\begin{align*}
\Big\{\sup_{t\in[0,\tau_R^{n,m}]}\big|Q_t^{n,m}\big|\ge\vv\Big\}=\Big\{\big|Q_{\tau_R^{n,m,\vv}}^{n,m}\big|\ge\vv\Big\}
\end{align*}
holds true, where
$$\tau_R^{n,m,\vv}:=\tau_R^{n,m}\wedge \inf\big\{t\ge0: \big|Q_t^{n,m}\big|\ge\vv\big\}.$$
Thus, \eqref{E13} can be obtained from the fact that
\begin{align*}
\lim_{R\to\8}\lim_{n\to\8}\P\big(\tau_R^{(n)}< T\big)+\lim_{R\to\8}\lim_{m\to\8}\P\big(\tau_R^{(m)}< T\big)=0,
\end{align*}
which is attainable  by invoking  Lemma \ref{lem2}, and provided that
\begin{align}\label{E6}
\lim_{R\to\8}\lim_{n,m\to\8}\P\big(\big|Q_{\tau_R^{n,m,\vv}}^{n,m}\big|\ge\vv \big)=0.
\end{align}

From \eqref{E2},
one obviously has for any $t\ge0,$
\begin{align*}
\d Q_t^{n,m}&=b_t^{n,m}\,\d t +\int_U f_t^{n,m}(z)\,\wt N(\d t,\d z),
\end{align*}
in which
$$b_t^{n,m}:=b_t(Y_{t_n}^{(n)})-b_t(Y_{t_m}^{(m)}) \quad \mbox{ and } \quad
f_t^{n,m}(z):=f_t(Y_{t_n}^{(n)},z)-f_t(Y_{t_m}^{(m)},z).$$
Applying  It\^o's formula followed by using \eqref{E5}  and $K_\cdot(R)>0$
yields that for   $0<t< \tau_R^{n,m,\vv}$,
\begin{align*}
&\d \Big(\e^{-2\int_0^tK_s(R)\,\d s}\big|Q_t^{n,m}\big|^2\Big)\\
&\le\e^{-2\int_0^tK_s(R)\,\d s}\Big(-2K_t(R)\big|Q_{t}^{n,m}\big|^2+
2\big\<Y_{t_n}^{(n)}-Y_{t_m}^{(m)},  b_t^{n,m}\big\>
+ \int_U |f_t^{n,m}(z) |^2\,\nu(dz)\Big)\,\d t\\
&\quad+
2\e^{-2\int_0^tK_s(R)\,\d s}\big\<p_{t
}^{(n)}-p_{t
}^{(m)},  b_t(Y_{t_n}^{(n)})-b_t(Y_{t_m}^{(m)})\big\>\,\d t+\d M_t^{n,m}\\
&\le K_t(R)\e^{-2\int_0^tK_s(R)\,\d s}\Big(-2\big|Q_{t }^{n,m}\big|^2+\big|Y_{t_n}^{(n)}-Y_{t_m}^{(m)}\big|^2\Big)\,\d t\\
& \quad+2\big|p_{t
}^{(n)}-p_{t
}^{(m)}\big|\cdot\big|b_t(Y_{t_n}^{(n)})-b_t(Y_{t_m}^{(m)})\big|\,\d t+\d M_t^{n,m}\\
&\le 2K_t(R)\e^{-2\int_0^tK_s(R)\,\d s} \big|p_{t
 }^{(n)}-p_{t
 }^{(m)}\big|^2\,\d t
+2\big|p_{t
 }^{(n)}-p_{t
 }^{(m)}\big|\cdot\big|b_t(Y_{t_n}^{(n)})-b_t(Y_{t_m}^{(m)})\big|\,\d t+\d M_t^{n,m}\\
&\le \Gamma(t,R)\big(|p_{t
}^{(n)}|+|p_{t
 }^{(m)}|\big)\,\d t+\d M_t^{n,m},
\end{align*}
where  $(M_t^{n,m})_{t\ge0}$ is a local martingale and
$
\Gamma(t,R):=4
RK_t(R) +4\sup_{ |x|\le R/2 }|b_t(x)|.
$
Subsequently, with the aid of $Y_0^{(n)}=Y_0^{(m)}=Y_0$ and $\tau_R^{n,m,\vv}\le \tau_R^{(n)}\wedge \tau_R^{(m)}\wedge T$,
we arrive at
\begin{equation}\label{e:addunique}
\begin{split}
\E \big|Q_{\tau_R^{n,m,\vv}}^{n,m}\big|^2 \le \e^{ 2\int_0^TK_s(R)\,\d s}\bigg(&\E\Big(\int_0^{\tau_R^{(n)}\wedge T}\Gamma(s,R)|p_{s-}^{(n)}|\,\d s\Big)\\
&+\E\Big(\int_0^{\tau_R^{(m)}\wedge T}\Gamma(s,R)|p_{s-}^{(m)}|\,\d s\Big)\bigg).
\end{split}
\end{equation}
Thus, \eqref{E6} follows from Chebyshev's inequality and by noting
\begin{align*}
\lim_{n\to\8}\E\Big(\int_0^{\tau_R^{(n)}\wedge T}\Gamma(s,R)|p_{s-}^{(n)}|\,\d s\Big)
+\lim_{n\to\8}\E\Big(\int_0^{\tau_R^{(m)}\wedge T}\Gamma(s,R)|p_{s-}^{(m)}|\,\d s\Big)=0,
\end{align*}
which  can indeed be established by repeating exactly the procedure to derive \eqref{E11}
(in particular, see the argument for \eqref{e:addcomment}).
\end{proof}

Now we turn to complete the proof of Proposition \ref{pro1}.
\begin{proof}[Proof of Proposition $\ref{pro1}$]  {\it Strong existence}.
Since  Lemma \ref{Lem3} is available, the proof on strong well-posedness of \eqref{E1} is more or less standard; see e.g. \cite[Theorem 3.1.1]{PR} and \cite[Theorem 1.5]{MSSZ} for more details. Nevertheless,  we herein provide a sketch to make the content self-contained.

Note that the space $L^\beta(\Omega, \mathscr D([0,T];\R^d))$ is complete w.r.t. locally uniform convergence in probability, where $\mathscr D([0,T];\R^d)$ stands for the set of   c\'{a}dl\'{a}g functions $f:[0,T]\to\R^d.$
Then,
Lemma \ref{Lem3}   implies that there is a c\'{a}dl\'{a}g, $(\mathscr F_t)_{t\ge0}$-adapted process $(Y_t)_{t\ge0}$ such that
\begin{align*}
\sup_{t\in[0,T]}|Y_t^{(n)}-Y_t|\to0 \quad \mbox{in probability as } n\to\8.
\end{align*}
Therefore, a subsequence, still written as $(Y_t^{(n)})_{t\in[0,T]}$, can be extracted so $\P$-a.s.,
\begin{align*}
\sup_{t\in[0,T]}|Y_t^{(n)}-Y_t|\to0 \quad  \mbox{ as } n\to\8.
\end{align*}
This thus implies that
\begin{align}\label{EE}
\sup_{t\in[0,T]} S_t<\8\quad \P\mbox{-a.s.}\quad   \mbox{ with } S_t:=\sup_{n\ge1}|Y_t^{(n)}|.
\end{align}

Define the following stopping time:   for any $R>0, $
\begin{align*}
\tau_R =T\wedge\inf\big\{t\in[0,T]:S_t>R\big\}.
\end{align*}
The continuity of $x\mapsto b_t(x)$, Assumption (${\bf H}_3$) and  the dominated convergence theorem yield that for $t\le \tau_R,$
\begin{align*}
\lim_{n\to\8}\int_0^tb_s(Y_{s_n}^{(n)}) \,\d s=\int_0^tb_s (Y_s ) \,\d s, \quad \P\mbox{-a.s.}
\end{align*}
Furthermore, applying  It\^o's isometry and combining the dominated convergence theorem with (${\bf H}_3$) and the continuity of $x\mapsto f_t(x,z)$ enable us to derive that
\begin{align*}
&\lim_{n\to\8}\E\Big|\int_0^{\tau_R}\int_Uf_s(Y_{s_n}^{(n)},z)-f_s(Y_s,z)  \wt N(\d s,\d z)\Big|^2\\
&=\lim_{n\to\8}\E\Big(\int_0^{\tau_R}\int_U\big|f_s(Y_{s_n}^{(n)},z)-f_s(Y_s,z)\big|^2  \nu(\d z)\d s\Big)=0.
\end{align*}
Then, for  a subsequence, still written as $(Y_t^{(n)})_{0\le t\le T}$,  we have  for any $t\le \tau_R,$
\begin{align*}
\lim_{n\to\8}\int_0^t \int_U f_s(Y_{s_n}^{(n)},z) \,\wt N(\d s,\d z)=\int_0^t\int_U f_s(Y_s,z) \,\wt N(\d s,\d z), \quad \P\mbox{-a.s.}
\end{align*}
Based on the preceding analysis, we conclude that $(Y_t)_{t\in[0,T]}$ is a strong solution to \eqref{E1} once we note  $\lim_{R\to\8}\tau_R=T$ by taking advantage of \eqref{EE}.

{\it Uniqueness}.
For any $R,T>0$, define  $\tau_R^{Y,T}=T\wedge \tau_R^Y$ with
$\tau_R^Y:= \inf\{t>0: |Y_t|>R\}$, where $(Y_t)_{t\ge0}$ is a strong solution to \eqref{E1}.  By taking $x=y$ in \eqref{E3}, we obviously obtain from (${\bf H}_2$) that
\begin{equation}\label{KL1}
\begin{split}
(\mathscr L_t^0V_\beta)(x,x)&\le \frac{1}{2}\beta (1+|x|^2)^{\frac{\beta}{2}-1} \big(2\<x,b_t(x)\>+\nu(|f_t(x,\cdot)|^2\I_U(\cdot))\big)\\
&\le \frac{1}{2}\beta\varphi(t)V_\beta(x),
\end{split}
\end{equation}
where  $\mathscr L_t^0$ was defined as in \eqref{KL}. Subsequently, by the aid of the locally integrable property of $\varphi$,
we can deduce that for some constant $  C(T)>0,$
\begin{equation}\label{addmoment}
\E\big|Y_{ \tau_R^{Y,T}}\big|^\beta\le C(T)(1+\E|Y_0|^\beta).
\end{equation}
This further implies that
\begin{align*}
R^\beta\P(\tau_R^Y<T)\le \E\big(\big|Y_{ \tau_R^{Y}}\big|^\beta\I_{\{\tau_R^Y<T\}}\big)\le  C(T)(1+\E|Y_0|^\beta).
\end{align*}
As a consequence,  $\tau_R^{X,T}\to T$, $\P$-a.s., as $R\to\8.$

Below, let $(Y_t)_{t\in[0,T]}$ and $(\tilde Y_t)_{t\in[0,T]}$ be two solutions to \eqref{E1} with the same initial value, and set
$\tau_R^*:=\tau_R^{Y,T}\wedge\tau_R^{\tilde Y,T}$. By following the line to derive
\eqref{e:addunique}, we deduce that for any $\vv>0$ and $t\in[0,T],$
\begin{align*}
 \E\Big(\e^{-\frac{1}{2}\beta\int_0^{t\wedge\tau_R^*}K_s(R)\,\d s}\big(\vv+|Y_{t\wedge\tau_R^*}-\tilde Y_{t\wedge\tau_R^*}|^2\big)^{\frac{\beta}{2}}\Big)=0
 \end{align*}
and so
\begin{align*}
 \E\Big(\e^{-\frac{1}{2}\beta\int_0^{t\wedge\tau_R^*}K_s(R)\,\d s}\big|Y_{t\wedge\tau_R^*}-\tilde Y_{t\wedge\tau_R^*}\big|^\beta\Big)=0.
 \end{align*}
This, combining
  $K_\cdot(R)>0$ with $\tau_R^{Y,T}\to T$ and $\tau_R^{\tilde Y,T}\to T$,  $\P$-a.s., as $R\to\8, $ and utilizing Fatou's lemma,  give that for any $t\in(0,T],$
 \begin{align*}
 \E|Y_t-\tilde Y_t|^\beta\le \liminf_{R\to\8}\E  \big|Y_{t\wedge\tau_R^*}-\tilde Y_{t\wedge\tau_R^*}\big|^\beta =0.
\end{align*}
Whence, the uniqueness of strong solution follows.

{\it Moment estimate}.
The establishment of \eqref{moment1}  can be done by examining the procedure to derive
\eqref{addmoment}, so we herein omit the corresponding details.
\end{proof}

With the help of Proposition  \ref{pro1},   we move to  apply the so-called  interlacing technique to construct the unique solution to the SDE \eqref{e} and thus verify Theorem \ref{thm}.

\begin{proof}[Proof of Theorem $\ref{thm}$]
Let for any $t\ge0,$
\begin{align}\label{EP}
Z^V_t=\int_0^t\int_VzN(\d s,\d z)\quad \mbox{ and } \quad D_p^V=\{t\ge0: Z^V_t\neq Z^V_{t-}, \triangle Z_t \in V\},
\end{align}
where $\triangle Z_t:=Z_t-Z_{t-}$, the increment of $Z_\cdot$ at time $t$.
Note that $D_p^V$ is a countable set so that it can be rewritten as
$
D_p^V=\{\si_1,\cdots,    \si_n,\cdots\},
$
where, $n\mapsto \si_n$ is increasing, and,
 for each fixed $n\ge1,$ $\si_n$ is a finite stopping time satisfying $\lim_{n\to\infty}\si_n=\infty$ a.s.  by taking  $\nu(\I_V)<\8$ into consideration. Let $p_n=\triangle Z_{\si_n}, n\ge1 $, i.e.,  the jump amplitude at the jumping time $\si_n$.
  Then, $(p_n)_{n\ge1}$ is an i.i.d sequence of random variables with the common distribution $\nu|_{V}/\nu(\I_V)$ and independent of $(\si_n)_{n\ge1}$.
Now, we set
\begin{equation*}
X_t^{(1)}=\left\{\begin{array}{cc}
                   Y_t, \qquad\qquad\qquad\,\, & \quad 0\le t<\si_1, \\
                   Y_{\si_1}+g_{\si_1-}(Y_{\si_1-},p_1), & t=\si_1,
                 \end{array}\right.\\
\end{equation*}
where $(Y_t)_{t\ge0}$ solves \eqref{E1}.
Obviously, the process $(X_t^{(1)})_{t\in[0,\si_1]}$ is   the unique solution to \eqref{e} on $[0,\si_1]$. Next, we set
\begin{equation*}
X_t^{(2)}=\left\{\begin{array}{cc}
                   X_t^{(1)}, \qquad\qquad\qquad\quad\,\,\, & \quad 0\le t\le \si_1, \\
                   Y_t+g_{\si_1-}(Y_{\si_1-},p_1), \qquad\quad & \quad \si_1<t<\si_2, \\
                   Y_{\si_2-}+\sum_{i=1}^2g_{\si_i-}(Y_{\si_i-},p_i), & t=\si_2,
                 \end{array}\right.\\
\end{equation*}
which is the unique solution to \eqref{e} on $[0,\si_2]$. Continuing  successively the previous procedure, the global solution $(X_t)_{t\ge0}$ to \eqref{e}  can be determined uniquely. In particular,   $(X_t)_{t\ge0}$ can be written  as follows:
$$X_t=Y_t+\sum_{i=1}^\infty\big(g_{\si_i-}(Y_{\si_i-},p_i)\I_{\{\si_i\le t\}}\big)=Y_t+\int_0^t\int_V g_s(Y_s,z)\,N(\d s,\d z).$$

Below, we proceed to prove the statement \eqref{E12}. In retrospect, $V_{\beta}(x)=(1+|x|^2)^{\frac{\beta}{2}}$ for all $x\in\R^d.$
Applying  It\^o's formula yields that
\begin{align*}
&\d\Big(\e^{-\int_0^t(\frac{\beta}{2}\varphi(s)+\nu(\I_V))\,\d s }V_\beta(X_t)\Big)\\
&= \e^{-\int_0^t(\frac{\beta}{2}\varphi(s)+\nu(\I_V))\,\d s} \Big(-\big(\beta\varphi(t)/2+\nu(\I_V)\big)V_\beta(X_t)+(\mathscr L_t V_\beta)(X_t) \Big)\,\d t+\d M_t,
\end{align*}
where $(M_t)_{t\ge0}$ is a martingale, and for $h\in C^2(\R^d)$
\begin{equation*}
(\mathscr L_th)(x):=(\mathscr L_t^0h)(x,x) +\nu\big( (h(x+g_t(x,\cdot))-h(x) )\I_V(\cdot) \big)
\end{equation*}
with $\mathscr L_t^0$ being given by \eqref{KL}.
Furthermore, by invoking the  inequality: $(a+b)^{\theta}\le a^\theta+b^\theta$, $a,b>0$ and $\theta\in(0,1],$ it follows that for any $x,y\in\R^d,$
\begin{align}\label{KL2}
V_\beta(x+y)-V_\beta(x)
  &\le\big| |y|^2+2\<x,y\>\big|^{\frac{\beta}{2}}
\le 2^{{\be}/{2}}|y|^\be+|x|^\be.
\end{align}
This, together with \eqref{KL1} and $({\bf H}_{2})$, leads to
\begin{equation*}
\begin{split}
(\mathscr L_tV_\beta)(x)
&\le \frac{\be}{2}(1+|x|^2)^{\be/2-1}\Big(2\<x,b_t(x)\>
 +\nu(|f_t(x,\cdot)|^2\I_U(\cdot))\\
 &\qquad\qquad \quad\quad\quad\quad\quad+\frac{2^{\be/2+1}}{\be}(1+|x|^2)^{1-\be/2}\nu( |g_t(x,\cdot)|^\be\I_V(\cdot))\Big)+|x|^\be\,\nu(\I_V)\\
 &\le  (\be\varphi(t) /2+\nu(\I_V) )V_\beta(x).
\end{split}
\end{equation*}
 Consequently, the desired assertion \eqref{E12} follows.
\end{proof}

Before the end of this section, we finish the proof of Theorem \ref{theorem1}, which is concerned with
a stronger moment estimate.

\begin{proof}[Proof of Theorem $\ref{theorem1}$]
According to Theorem \ref{thm}, under Assumptions  $({\bf H}_1 )$, $({\bf H}_{2}')$ and $({\bf H}_{3})$, the SDE \eqref{e} has a unique strong solution.

 Recall that $D_p^V=\{\si_1,\cdots,    \si_n,\cdots\},$
and $(p_n)_{n\ge1}$ is the   Poisson point process associated with the L\'evy process $(Z_t^V)_{t\ge0}$ given in \eqref{EP}. More details on $(\si_n)_{n\ge1}$ and
$(p_n)_{n\ge1}$ are presented in the beginning part of the proof of  Theorem \ref{thm}.
Obviously, for any $T>0$, we have
 \begin{align*}
 \E\Big(\sup_{0\le t\le T}|X_t|^\beta\Big)
 &\le \E\Big(\I_{\{0\le T<\si_1\}}  \sup_{0\le t\le T}|X_t|^\beta  \Big) +\sum_{n=1}^\8\E\Big(\I_{\{\si_n\le T<\si_{n+1}\}} \sum_{k=0}^{ n-1 }  \sup_{ \si_k\le t  <\si_{k+1} }|X_t|^\beta  \Big)\\
 &\quad+\sum_{n=1}^\8\E\Big(\I_{\{\si_n\le T<\si_{n+1}\}} \sup_{ \si_n\le t\le T}|X_t|^\beta  \Big) \\
 &=:\Gamma_1+\Gamma_2+\Gamma_3.
 \end{align*}
Therefore, to achieve \eqref{moment}, it is sufficient to show that  for some constant  $C_T^*>0$,
\begin{align}\label{ER2}
\Gamma_i\le C_T^*(1+\E|X_0|^\beta),\quad i=1,2,3.
\end{align}

For the validity  of \eqref{ER2}, we firstly verify that there exists a constant $C_T^{**}>0$ such that for any   $0\le k\le (n-1)^+$ and $\si_n\le T$,
\begin{align}\label{ER}
\E\Big(\sup_{\si_k\le s<\si_{k+1}}|X_s|^\beta\big|\mathscr G_{n,k}\Big)\le C_T^{**}\big(1+   |X_{\si_k}|^\beta\big),
\end{align}
 where $\si_0=0,$ and
$$
\mathscr G_{n,k }:=\si(\si_i,1\le i\le n+1)\vee \mathscr F_{T\wedge\si_k}, \quad  n,k\ge0.
$$
By It\^o's formula, it follows from $({\bf H}_2')$ that for any $t\in[\si_k,\si_{k+1})$ with
$0\le k\le (n-1)^+$,
\begin{equation}\label{EW2}
\begin{split}
\E\Big(\sup_{ \si_k\le s\le t}|X_s|^2\big|\mathscr G_{n,k}\Big)&\le  |X_{\si_k} |^2+3\int_{\si_k}^t\phi(s)\big(1+\E(|X_s|^2 |\mathscr G_{n,k})\big)\,\d s\\
&\quad+\E\Big(\sup_{\si_k\le s\le t} |M_{k,s}|\big|\mathscr G_{n,k }\Big),
\end{split}
\end{equation}
where
\begin{align*}
 M_{k,t}:=\int_{\si_k^V}^t\int_U \big(2\<X_{s-},f_s(X_{s-},z)\>+|f_s(X_{s-},z)|^2\big) \tilde N(\d s,\d z).
\end{align*}
Next, by applying BDG's inequality (see e.g. \cite[Theorem 1]{MR}), in addition to (${\bf H}_2'$), there exist constants $c_1,c_2>0$ such that for any for any $t\in[\si_k,\si_{k+1})$ with
$0\le k\le (n-1)^+$,
\begin{align*}
\E\Big(\sup_{\si_k\le s\le t} |M_{k,s}|\big|\mathscr G_{n,k }\Big)&\le c_1\E\bigg(\bigg(\int_{\si_k}^t\int_U|X_s|^2|f_s(X_{s-},z)|^2\,\nu(\d z)\,\d s\bigg)^{{1}/{2}}\big|\mathscr G_{n,k }\bigg)\\
&\quad+c_1\int_{\si_k}^t\int_U\E\big( |f_s(X_{s-},z)|^2\big|\mathscr G_{n,k }\big)\,\nu(\d z)\,\d s \\
&\le \frac{1}{2}\E\Big(\sup_{ \si_k\le s\le t}|X_s|^2\big|\mathscr G_{n,k}\Big)+c_2\int_{\si_k}^t\phi(s)\big(1+\E(|X_s|^2 |\mathscr G_{n,k})\big)\,\d s.
\end{align*}
Subsequently, plugging the estimate above  back into \eqref{EW2} followed by Gronwall's inequality yields that there exists a constant $c_3>0$ such that   for any   $0\le k\le (n-1)^+$ and $\si_n\le T$,
\begin{align*}
\E\Big(\sup_{ \si_k\le s<\si_{k+1} }|X_s|^2\big|\mathscr G_{n,k}\Big)\le c_3\e^{c_3\int_0^T\phi(u)\d u}\big(1+ |X_{\si_k} |^2\big).
\end{align*}
This, combining with Jensen's inequality, implies  \eqref{ER}.

Let $N_t=\sharp\{s\in [0,t]: Z^V_s\neq Z^V_{s-}\}$, which is a Poisson process with the intensity $\nu(\I_V)t$ (so $\E N_T=\nu(\I_V)T$). Making use of the definition of $\mathscr G_{n,k}$ and \eqref{ER} gives us that
\begin{align*}
\Gamma_2&=\sum_{n=1}^\8\E\Big(\I_{\{\si_n\le T<\si_{n+1}\}} \sum_{k=0}^{ n-1 }  \E\Big(\sup_{\si_k\le t  <\si_{k+1} }|X_t|^\beta\big| \mathscr G_{n,k} \Big)\Big)\\
&\le  C_T^{**}\sum_{n=1}^\8\E\Big(\I_{\{\si_n\le T<\si_{n+1}\}} \sum_{k=0}^{ n-1 }    \big(1+    |X_{\si_k}|^\beta \big)\Big)\\
&\le C_T^{**}\Big(1+    \sup_{0\le t\le T}\E|X_t|^\beta  \Big)\sum_{n=1}^\8\P(N_T=n) n
\\
&\le  \nu(\I_V)T C_T^{**}\Big(1+    \sup_{0\le t\le T}\E|X_t|^\beta  \Big),
\end{align*}
where in the second identity we exploited the fact that  $\I_{\{\si_n\le T<\si_{n+1}\}}$ is independent of $\sum_{k=0}^{ n-1 }     (1+    |X_{\si_k}|^\beta  )$. Hence, \eqref{ER2} with $i=2$ is available by the aid of \eqref{E12}.

On the other hand,  by following exactly the strategy to derive \eqref{ER},  we can obtain that for   $T\in[\si_n,\si_{n+1}),$
  \begin{align*}
\E\Big(\sup_{ \si_n\le t\le T  }|X_t|^\beta\big|\mathscr G_{n,n}\Big)\le C_T^{**}\big(1+   |X_{\si_n}|^\beta\big).
\end{align*} In particular, it is ready  to see  that
\begin{align*}
\Gamma_1&\le \E\Big(  \E\Big(\sup_{ 0\le t\le T}|X_t|^\beta\big|\mathscr G_{0,0}  \Big)\Big) \le C_T^{**} (1+   \E|X_0|^\beta )
\end{align*}
so that \eqref{ER2} holds true for $i=1.$
Furthermore,  we can obtain  that
\begin{align*}
 \Gamma_3&=\sum_{n=1}^\8\E\Big(\I_{\{\si_n\le T<\si_{n+1}\}}\E\Big( \sup_{ \si_n\le t\le T}|X_t|^\beta  \big|\mathscr G_{n,n}\Big) \Big)\\
 &\le C_T^{**}\sum_{n=1}^\8\E\Big(\I_{\{\si_n\le T<\si_{n+1}\}} \big(1+   |X_{\si_n}|^\beta\big) \Big)\\
 &\le 2C_T^{**}\sum_{n=1}^\8 \big(\P(N_T=n) \big(1+   \E|X_{T\wedge\si_n-}|^\beta +\E|g_{T\wedge\si_n-}(X_{T\wedge\si_n-},p_n)|^\beta\big) \big)\\
&\le  2C_T^{**}\big(1+\phi(T)^{\frac{\beta}{2}}\big) \sum_{n=1}^\8 \big(\P(N_T=n) \big(1+   \E|X_{T\wedge\si_n-}|^\beta \big) \big)\\
&\le 2C_T^{**}\big(1+\phi(T)^{\frac{\beta}{2}}\big)\Big(1+    \sup_{0\le t\le T}\E|X_t|^\beta  \Big),
\end{align*}
where in the second inequality we employed that
$X_{\si_n} = X_{\si_n-} +g_{\si_n-}(X_{\si_n-},p_n)$ and the strong Markov property,
 and in the third inequality we took advantage of $({\bf H}_2')$ and the non-decreasing property of $\phi $ as well as the fact that $(a+b)^{\beta/2}\le a^{\beta/2}+b^{\beta/2}$ for all $a,b\ge0$. Finally, \eqref{ER2} with $i=3$ is verifiable on the basis of \eqref{E12}. Therefore, the proof is finished.
\end{proof}

\section{Proofs of Theorem \ref{theorem2},  Theorem  \ref{POC} and Remark $\ref{EX}$ {\rm(ii)}}\label{sec3}
\subsection{Proof of
Theorem \ref{theorem2}}

Based on the warm-up Theorem \ref{thm}, along with the fixed point theorem,
in the following part  we aim to complete the proof of Theorem \ref{theorem2}.

\begin{proof}[Proof of Theorem $\ref{theorem2}$]
For a  fixed   horizon $T>0$,   define the following path space:
$$
\scr{C}^{X_0}_{T}=\Big\{\mu\in C([0,T];\scr{P}_\be(\R^{d})): \sup_{t\in[0,T]}\mu_t(|\cdot|^\beta)<\8, \mu_0=\scr{L}_{X_0}\Big\},
$$
where  $X_0\in L^\be(\Omega\to\R^d,\scr{F}_0,\P)$ is the initial value of $(X_t)_{t\ge0}$, and
 \begin{align*}
C([0,T];\scr{P}_\beta(\R^{d})):=\{\mu:[0,T]\to\scr{P}_\beta(\R^{d}) \mbox{ is weakly continuous}\}.
\end{align*}
For $\gamma>0,$ $(\scr{C}^{X_0}_{T}, \W_{\be,\ga})$ is a complete metric space, where
$$
\W_{\be,\ga}(\mu,\widetilde{\mu}):=\sup_{0\le t\le T}\big(\e^{-\ga t}\W_\be(\mu_t,\widetilde{\mu}_t)\big), \quad \mu,\widetilde{\mu}\in\scr{C}^{X_0}_{T}.
$$

In the sequel, we work with  the  decoupled SDE associated with \eqref{MV}:  for any $\mu\in \scr{C}^{X_0}_{T}$,
\begin{equation}\label{df}
\d X^{\mu}_t=b(X^{\mu}_t,\mu_t)\,\d t+\int_U f(X^\mu_{t-},\mu_t,z)\,\wt N(\d t,\d z)+\int_V g(X_{t-}^\mu,\mu_t,z)\,N(\d t,\d z)
\end{equation}
with the initial value $X^\mu_0=X_0$. By setting for any $t\ge0,$ $x,z\in\R^d$ and $\mu\in\scr{C}^{X_0}_{T},$
$$b_t^\mu(x):=b(x,\mu_t), \quad f_t^\mu(x,z):=f(x,\mu_t,z) \quad  \mbox{ and } \quad   g_t^\mu(x,z):=g(x,\mu_t,z),$$
the SDE \eqref{df} can be reformulated  so it fits into the framework  \eqref{e}. From \eqref{EQ}, it is easy to see that for any $t\ge0$, $x,y\in\R^d$ and  $\mu\in \scr{C}^{X_0}_{T}$,
\begin{align*}
2\<b_t^\mu(x)-b_t^\mu(y),x-y\>+\nu(|f_t^\mu(x,\cdot)-f_t^\mu(y,\cdot)|^2\I_U(\cdot))\le L_1|x-y|^2
\end{align*}
so that \eqref{E5} holds true with $b_t$ and $f_t$ being replaced by
 $b_t^\mu$ and $f_t^\mu$, respectively. In addition, by virtue of \eqref{EQ*} and (${\bf A}_2$), it follows
  that for some constants $c_1,c_2,c_3>0,$
\begin{equation}\label{EY}
\begin{split}
2&\<x,b_t^\mu(x)\>
  +\nu( |f_t^\mu(x,\cdot)|^2\I_U(\cdot)) + 2^{\frac{\be}{2}+1} \be^{-1}(1+|x|^2)^{1-\frac{\be}{2}}  \nu( |g_t^\mu(x,\cdot)|^\be\I_V(\cdot)) \\
  &\le L_2(1+|x|^2+\mu_t(|\cdot|^\beta)^{\frac{2}{ \beta}})\\
  &\quad+
  c_1(1+|x|^2)^{1-\frac{\be}{2}}\big(\nu((1+|\cdot|^\beta)\I_V(\cdot))(|x|^\beta+\mu_t(|\cdot|^\beta))
  +\nu(|g({\bf0},\delta_{\bf0},\cdot)|^\beta\I_V(\cdot))\big)\\
  &\le c_2(1+|x|^2+\mu_t(|\cdot|^\beta)^{\frac{2}{ \beta}})\\
  &\le c_3\Big(1+\sup_{0\le s\le t}\mu_s(|\cdot|^\beta)^{\frac{2}{ \beta}}
  \Big)(1+|x|^2),\quad 0\le t\le T,
\end{split}
\end{equation}
 where
 we used \eqref{e:welldefined}
 in the first inequality,
   and the second inequality is valid due to Young's inequality. Therefore, (${\bf H}_2$) holds true thanks to $\mu\in\scr{C}^{X_0}_{T}$.
 Furthermore, $({\bf A}_3)$    implies  (${\bf H}_3$) directly. Consequently,
 according to Theorem \ref{thm},  \eqref{df} has the unique strong solution $(X^{\mu}_t)_{t\ge0}$ under
 $({\bf A}_1)$-$({\bf A}_{3})$, along with \eqref{e:welldefined}.

Now, we
define a mapping $\scr{C}^{X_0}_{T}\ni\mu\mapsto\Phi(\mu)$ by
\begin{equation}\label{Phi}
(\Phi(\mu))_t=\scr{L}_{X^{\mu}_t}, \quad t\in[0,T].
\end{equation}
In the sequel, we shall claim  respectively that
(i) $\Phi:\scr{C}^{X_0}_{T}\to\scr{C}^{X_0}_{T}$, and
(ii)  $\Phi$ is contractive under $\W_{\be,\ga}$  for some $\gamma>0$ large enough.
Once  (i) and (ii) are available, the classical Banach fixed point theorem yields that the map $\Phi$ defined in \eqref{Phi} has a unique fixed point, still written as  $\mu$, so  $(\Phi(\mu))_t=\mu_t=\scr{L}_{X^{\mu}_t}$ for any $t\in[0,T]$. Thus, the measure variable $(\mu_t)_{t\in[0,T]}$ in \eqref{df} can be replaced by $(\scr{L}_{X^{\mu}_t})_{t\in[0,T]}$. Accordingly, \eqref{MV} has a unique strong solution.

 For the validity of (i), we need to   show that
there is a constant $C_{T,\mu}>0$ such that
\begin{equation}\label{sup}
\sup_{0\le t\le T}\E|X^{\mu}_t|^\be\le C_{T,\mu}(1+\E|X_0|^\be),\quad \mu\in\scr{C}^{X_0}_{T}, T>0.
\end{equation}
and that
\begin{align}\label{WR}
\mathscr L_{X_t^\mu}\overset{w}{\longrightarrow}\mathscr L_{X_0^\mu} \quad  \mbox{ as  } t \to0.
\end{align}
Indeed,  \eqref{sup} follows from \eqref{E12} and
 \eqref{EY}.
For any $R>0$ and $h\in \hbox{Lip}_b(\R^d)$ (i.e., the set of bounded Lipschitz functions on $\R^d$),    one obviously has for any $t\in[0,T]$
\begin{align*}
\E|h(X_t^\mu)-h(X_0^\mu)|\le \|h\|_{\hbox{Lip}}\E | X_{t\wedge\tau_R}^\mu - X_0^\mu | +2\|h\|_\8\P( \tau_R\le T)=:I_1(t,R)+I_2(T,R),
\end{align*}
where $\tau_R:=\{t>0: |X_t^\mu|>R\}$ and $\|h\|_{\mbox{Lip}}$ is the Lipschitz constant of $h$.
Based on this, to verify \eqref{WR}, it suffices to show that $\lim_{t\to0}I_1(t,R)=0$ for fixed $R$ and $\lim_{R\to\8}I_2(T,R)=0$, respectively.
By It\^o's formula, in addition to the local boundedness (see $({\bf A}_1)$ and $({\bf A}_3)$ for more details),
it follows that for any $\mu\in\scr{C}^{X_0}_{T}$,  $\vv>0,$ and $t\in[0,T],$
\begin{align*}
&\E(\vv+|X_{t\wedge\tau_R}^\mu-X_0^\mu|^2)^{\frac{\beta}{2}}\\
&\le \E\bigg(\int_0^{t\wedge\tau_R}\Big[\frac{1}{2}\beta(\vv+|X_s^\mu-X_0^\mu|^2)^{ \frac{\beta}{2}-1}\\
&\quad\quad\quad\times\Big(2\<X_s^\mu-X_0^\mu,b(X^{\mu}_s,\mu_s)\>+\nu( |f(X^{\mu}_s,\mu_s,\cdot)|^2\I_U(\cdot))\\
&\qquad\qquad\quad +2^{\frac{\be}{2}+1} \be^{-1}(\vv+|X_s^\mu-X_0^\mu|^2)^{1-\frac{\be}{2}}  \nu( |g (X^{\mu}_s,\mu_s,\cdot)|^\be\I_V(\cdot))\Big)\Big]\,\d s\bigg)\\
&\quad+ \nu(V)\E\bigg(\int_0^{t\wedge\tau_R}|X_{s\wedge\tau_R}^\mu-X_0^\mu|^{\beta}\,\d s\bigg)\\
&\le C_{\vv,R}t,
\end{align*}
where $C_{\vv,R}$ is a positive constant depending on the parameters $\vv,R.$
See the arguments in the end of the proof for Theorem \ref{thm}.
The preceding  estimate, besides Jensen's inequality for $\beta\in[1,2]$,  implies that for any $\vv,R>0,$
\begin{align*}
\E | X_{t\wedge\tau_R}^\mu - X_0^\mu |\le \E(\vv+|X_{t\wedge\tau_R}^\mu-X_0^\mu|^2)^{\frac{1}{2}}\le  (C_{\vv,R}t)^{\frac{1}{\beta}}.
\end{align*}
This apparently  leads to   $\lim_{t\to0}I_1(t,R)=0.$
Furthermore,
  $\lim_{R\to\8}I_2(T,R)=0$ can be handled by following the argument to derive \eqref{addmoment} and using \eqref{EY}.
Indeed, we have
\begin{align*}
\P(\tau_R\le T)\le \frac{\E\big(1+\big|X_{
T\wedge \tau_R}^\mu\big|^2\big)^{\beta/2}}{(1+
R^2)^{\beta/2}}.
\end{align*}
Then, $\lim_{R\to\8}I_2(T,R)=0$ is available by taking \eqref{sup} into consideration.

We turn  to show  (ii). Recall that $(Z^V_t)_{t\ge0}$ is defined as in \eqref{EP}. Furthermore,
 $ (p_n)_{ n\ge1} $ (i.e., the sequence concerning jumping amplitude of $(Z^V_t)_{t\ge0}$) are i.i.d   random variables with the common distribution $\nu|_{V}/\nu(V)$ and independent of $(\si_n)_{n\ge1}$ (i.e., the sequence of jumping time of $(Z^V_t)_{t\ge0}$); see the first paragraph of the proof of Theorem \ref{thm} for further details.
Note  obviously that
\begin{align*}
\mathbb W_{\beta,\gamma}(\Phi(\mu),\Phi(\tilde \mu))^{\beta}
&\le\sup_{0\le t\le T}\Big(\e^{-\beta\gamma t} \E\big(\I_{\{0\le  t< \si_1\}}\E\big(|\Upsilon_t^{\mu,\tilde\mu}|^\beta\big|\mathscr G_0\big) \Big)\Big)\\
&\quad+\sup_{0\le t\le T}\Big(\e^{-\beta\gamma t} \sum_{n=1}^\8\E\big(\I_{\{\si_n\le  t<\si_{n+1} \}}\E\big(|\Upsilon_t^{\mu,\tilde\mu}|^\beta\big|\mathscr G_n\big) \big)\Big)\\
&=\Gamma_1(\mu,\tilde\mu)+\Gamma_2(\mu,\tilde\mu),
\end{align*}
where
$$\Upsilon_t^{\mu,\tilde\mu}: =X^{\mu}_t-X^{\tilde\mu}_t \quad \mbox{ and } \quad \mathscr G_n:=\si\{\si_i, 1\le i\le n+1\}\vee \mathscr F_{\si_n}.$$
Thus, to obtain (ii), it remains  to show that there are constants $C_1(T),C_2(T)>0$ such that
\begin{align}\label{EY1}
\Gamma_1(\mu,\tilde\mu)\le  C_1(T)/\gamma^{\frac{\beta}{2}} \mathbb W_{\beta,\gamma}( \mu , \tilde \mu )^{\beta},
\end{align}
and
\begin{align}\label{EY2}
\Gamma_2(\mu,\tilde\mu)\le C_2(T)\big(1/\gamma+ 1/\gamma^{\frac{\beta}{2}}\big)\mathbb W_{\beta,\gamma}( \mu , \tilde \mu )^\beta.
\end{align}

Let $\si_0=\si_{0-}=0.$
For the establishments of \eqref{EY1} and \eqref{EY2}, we start to verify that  for any $n\ge0$ and  $t\in[\si_n,\si_{n+1})$,
\begin{align}\label{ET2-}
 \E(|\Upsilon_t^{\mu,\tilde\mu}|^\beta|\mathscr G_n)\le  \bigg(|\Upsilon_{\si_n}^{\mu,\tilde\mu}|^\beta +L_1^{\frac{\beta}{2}}\Big(\int_{\si_n}^{ t } \mathbb W_\beta(\mu_s,\tilde\mu_s)^2\d s\Big)^{\frac{\beta}{2}}\bigg)\e^{L_1(t-\si_n)}.
\end{align}
Indeed, applying It\^o's formula, in addition to \eqref{EQ},  we obtain   that for any  $t\in[\si_n,\si_{n+1}),$
\begin{align*}
\d |\Upsilon_t^{\mu,\tilde\mu}|^2
&\le L_1\big(|\Upsilon_t^{\mu,\tilde\mu}|^2+\mathbb W_\beta(\mu_t,\tilde\mu_t)^2\big)\,\d t+  \d \bar{M}_t,
\end{align*}
where $\bar{M}_\cdot$ is a martingale. Thus, Gronwall's inequality yields that for any $t\in[\si_n,\si_{n+1}),$
\begin{align*}
 \E(|\Upsilon_t^{\mu,\tilde\mu}|^2|\mathscr G_n)\le  \Big(|\Upsilon_{\si_n}^{\mu,\tilde\mu}|^2 +L_1\int_{\si_n}^{ t } \mathbb W_\beta(\mu_s,\tilde\mu_s)^2\d s\Big)\e^{L_1(t-\si_n)}.
\end{align*}
Whence, \eqref{ET2-} follows from Jensen's inequality.

By means of \eqref{ET2-}, in addition to $X_0^\mu=X_0^{\tilde \mu}$, we have
\begin{align*}
\Gamma_1(\mu,\tilde\mu)&\le L_1^{\frac{\beta}{2}}\e^{L_1T} \sup_{0\le t\le T}\Big(\e^{-\beta\gamma t}  \Big(  \int_0^{ t }  \mathbb W_\beta(\mu_s,\tilde\mu_s)^2\d s\Big)^{\frac{\beta}{2}} \Big)\\
&\le (L_1/(2\gamma))^{\frac{\beta}{2}}\e^{L_1T} \mathbb W_{\beta,\gamma}( \mu , \tilde \mu )^{\beta}.
\end{align*}
Hence, \eqref{EY1} follows right now. Again, by virtue of \eqref{ET2-}, we find that
\begin{align*}
\Gamma_2(\mu,\tilde\mu)&\le\e^{L_1T}\sup_{0\le t\le T}\Big(\e^{-\beta\gamma t} \sum_{n=1}^\8\E\big(\I_{[\si_n,\si_{n+1})}(t)\Big(|\Upsilon_{\si_n}^{\mu,\tilde\mu}|^\beta +L_1^{\frac{\beta}{2}}\Big(\int_{\si_n}^{ t } \mathbb W_\beta(\mu_s,\tilde\mu_s)^2\d s\Big)^{\frac{\beta}{2}}   \Big)\\
&\le \e^{L_1T}\sup_{0\le t\le T}\bigg(\e^{-\beta\gamma t} \sum_{n=1}^\8\E\big(\I_{[\si_n,\si_{n+1})}(t) |\Upsilon_{\si_n}^{\mu,\tilde\mu}|^\beta\big)\bigg)+(L_1/(2\gamma))^{\frac{\beta}{2}}\e^{L_1T} \mathbb W_{\beta,\gamma}( \mu , \tilde \mu )^{\beta}\\
&=:\e^{L_1T} \hat \Gamma_2(\mu,\tilde\mu)+(L_1/(2\gamma))^{\frac{\beta}{2}}\e^{L_1T} \mathbb W_{\beta,\gamma}( \mu , \tilde \mu )^{\beta}.
\end{align*}
 Next, because of
\begin{align}\label{ET4}
|\Upsilon_{\si_n }^{\mu,\tilde\mu}|\le |\Upsilon_{\si_n- }^{\mu,\tilde\mu}|+ |g(X_{\si_n-}^\mu,\mu_{\si_n},p_n)-g(X_{\si_n-}^{\tilde\mu},\tilde\mu_{\si_n },p_n)|,
\end{align}
it follows from \eqref{EQ*}
that
\begin{equation}\label{ET5}
\begin{split}
 |\Upsilon_{\si_n }^{\mu,\tilde\mu}|^\beta
&\le 2  |\Upsilon_{\si_n- }^{\mu,\tilde\mu}|^\beta+8L_1^\beta\big(|\Upsilon_{\si_n- }^{\mu,\tilde\mu}|^\beta+\mathbb W_\beta(\mu_{\si_n },\tilde\mu_{\si_n })^\beta\big)(1+|p_n|^\beta)\\
&\le \xi_n \big( |\Upsilon_{\si_n- }^{\mu,\tilde\mu}|^\beta  + \mathbb W_\beta(\mu_{\si_n },\tilde\mu_{\si_n })^\beta\big),
\end{split}
\end{equation}
where $\xi_n:=2+8 L_1^\beta (1 + |p_n|^\beta))$,
and $\sigma_0=\sigma_{0-}=0$. Due to the stationarity of $(p_n)_{t\ge0}$, one has $c_*:=2+8 L_1^\beta (1 + \E|p_n|^\beta)$
(which is independent of $n$).
Furthermore, applying \eqref{ET2-} and
 \eqref{ET5} repeatedly  yields that
\begin{align*}
 \E\big( \big|\Upsilon_{ \si_n-}^{\mu,\tilde\mu}\big|^\beta\I_{[\si_n,\si_{n+1})}(t)\big) &=\E\big|\Upsilon_{ \si_n-}^{\mu,\tilde\mu}\big|^\beta\E\I_{[\si_n,\si_{n+1})}(t)\\
 &\le \e^{L_1T}\E\bigg(\Big(\big|\Upsilon_{\si_{n-1}}^{\mu,\tilde\mu}\big|^\beta + L_1^{\frac{\beta}{2}}  \Big(\int_{\si_{n-1}}^{\si_n} \mathbb W_\beta(\mu_s,\tilde\mu_s)^2\d s\Big)^{\frac{\beta}{2}}\Big)\I_{[\si_n,\si_{n+1})}(t)\bigg)\\
 &\le \e^{L_1T}\E\bigg(\bigg(\xi_{n-1} \big( |\Upsilon_{\si_{n-1}- }^{\mu,\tilde\mu}|^\beta  + \mathbb W_\beta(\mu_{\si_{n-1} },\tilde\mu_{\si_{n-1} })^\beta\big)\\
 &\qquad\qquad\quad\quad+L_1^{\frac{\beta}{2}}   \Big(\int_{\si_{n-1}}^{\si_n} \mathbb W_\beta(\mu_s,\tilde\mu_s)^2\d s\Big)^{\frac{\beta}{2}}\bigg)\I_{[\si_n,\si_{n+1})}(t)\bigg)\\
 &= \e^{L_1T}\E\bigg(\bigg(c_* \big( |\Upsilon_{\si_{n-1}- }^{\mu,\tilde\mu}|^\beta  + \mathbb W_\beta(\mu_{\si_{n-1} },\tilde\mu_{\si_{n-1} })^\beta\big)\\
 &\qquad\qquad\quad\quad+L_1^{\frac{\beta}{2}}   \Big(\int_{\si_{n-1}}^{\si_n} \mathbb W_\beta(\mu_s,\tilde\mu_s)^2\d s\Big)^{\frac{\beta}{2}}\bigg)\I_{[\si_n,\si_{n+1})}(t)\bigg)\\
 &\le \e^{L_1T}\E\bigg[\bigg(c_*\e^{L_1T}\Big(c_* \big( |\Upsilon_{\si_{n-2}- }^{\mu,\tilde\mu}|^\beta  + \mathbb W_\beta(\mu_{\si_{n-2} },\tilde\mu_{\si_{n-2} })^\beta\big) \\
 &\qquad\qquad\quad\quad\quad\quad+  L_1^{\frac{\beta}{2}} \Big(\int_{\si_{n-2}}^{\si_{n-1}}\mathbb W_\beta(\mu_s,\tilde\mu_s)^2\d s\Big)^{\frac{\beta}{2}}\Big)\\
 &\qquad \quad+c_* \mathbb W_\beta(\mu_{ \si_{n-1} },\tilde\mu_{ \si_{n-1} })^\beta+L_1^{\frac{\beta}{2}} \Big(\int_{\si_{n-1}}^{\si_n} \mathbb W_\beta(\mu_s,\tilde\mu_s)^2\d s\Big)^{\frac{\beta}{2}}\bigg)\I_{[\si_n,\si_{n+1})}(t)\bigg]\\
 &\le\cdots\\
 &\le \e^{L_1T} (c_{*} \e^{L_1T})^{n-1}
 \sum_{i=0}^{n-1}\Big(c_* \E\big( \mathbb W_\beta(\mu_{\si_{i} },\tilde\mu_{\si_{i} })^\beta\I_{[\si_n,\si_{n+1})}(t)\big)\\
  &\qquad\qquad\qquad\quad\qquad\qquad+L_1^{\frac{\beta}{2}} \E\Big(\int_0^t \mathbb W_\beta(\mu_s,\tilde\mu_s)^2\d s\Big)^{\frac{\beta}{2}}\P (N_t=n)\Big),
\end{align*}
where in the first identity we used the independence
of $ \big|\Upsilon_{ \si_n-}^{\mu,\tilde\mu}\big|^\beta$ and $\I_{[\si_n,\si_{n+1})}(t)$; the first inequality is valid by invoking \eqref{ET2-} and the
independent property before as well as the dependence between $\int_{\si_{n-1}}^{\si_n} \mathbb W_\beta(\mu_s,\tilde\mu_s)^2\d s$ and $\I_{[\si_n,\si_{n+1})}(t)$; the second inequality holds true from \eqref{ET5};  in the second identity  we employed that $\xi_{n-1}$ is independent of $ |\Upsilon_{\si_{n-1}- }^{\mu,\tilde\mu}|^\beta\I_{[\si_n,\si_{n+1})}(t)$
and that $(\si_n)_{n\ge1}$ is independent of $(p_n)_{n\ge1}$; in the last inequality
we  used $c_*>1$ and
 $\Upsilon_0^{\mu,\tilde\mu}={\bf 0}$ due to $X_0^\mu=X_0^{\tilde\mu}=X_0.$ Note that $\P(N_t=n)=\e^{-\lambda t}(\lambda t)^n/n!$ for $\lambda:=\nu(\I_V)$ and
recall from \cite[p.\ 8]{CT} that
\begin{align*}
\sum_{k=1}^n\E(\mathbb W_\beta(\mu_{\si_k},\tilde\mu_{\si_k})^\beta|N_t=n)=\frac{n}{t}\int_0^t\mathbb W_\beta(\mu_s,\tilde\mu_s)^\beta\,\d s.
\end{align*}
Subsequently, we find that there exists a constant $C_3(T)>0$ such that
\begin{align*}
\hat \Gamma_2(\mu,\tilde\mu)\le C_3(T)\big(1/\gamma+ 1/(\gamma)^{\frac{\beta}{2}}\big)\mathbb W_{\beta,\gamma}( \mu , \tilde \mu )^\beta.
\end{align*}
This thus implies \eqref{EY2}.
Based on the preceding analysis,  we conclude that $\Phi$ is contractive by choosing  $\gamma>0$ large enough so the statement (ii) follows.

Furthermore, according to \eqref{EQ*} and \eqref{e:welldefined},
$$ \big(\nu(|g(x,\mu,\cdot)|^\beta\I_V(\cdot) )\big)^{\frac{2}{\beta}} \le c_0 (1+|x|^2+\mu(|\cdot|^\be)^{\frac{2}{\beta}} ).$$ This along with \eqref{A21} yields that Assumption (${\bf H}_2'$) is satisfied for the decoupled SDE \eqref{df}. Therefore,
in terms of Theorem \ref{theorem1}, there exists a constant $c_T>0$ such that
$$
\E\Big(\sup_{t\in[0,T]}|X_t^\mu|^\be\Big)\le c_{T}\Big(1+\E|X_0|^\be+\sup_{0\le t\le T}\mu_t(|\cdot|^\beta)\Big).
 $$
Thus, the assertion \eqref{EQ-} follows by  noting that $(X_t)_{t\ge0}$  and $(X_t^{\mu})_{t\ge0}$ with the alternative $\mu_t=\mathscr L_{X_t}$ share the same law on the path space $C([0,T];\R^d)$, and by making use of \eqref{B0}.
\end{proof}

Before the end of this section, we make some  comments.

\begin{remark}\label{Add:re1}\rm
\begin{itemize}
\item[{\rm(i)}]
The proof above is inspired essentially by that of \cite[Theorem 1]{CT} with some essential  modifications.
It is quite natural to  directly derive, via an approximate argument and  It\^o's  formula, that $\Phi$ constructed in the proof of Theorem \ref{theorem2} is contractive. Nevertheless,   some issues might be  encountered when such a direct approach is adopted.
To demonstrate the underlying difficulty, we set $f(x,\mu,z)=f(z)$ and $g(x,\mu,z)\equiv{\bf0}$ for simplicity. Thus, the chain rule, together with \eqref{EQ},  shows formally that for $\beta\in[1,2]$,
\begin{align*}
\d |\Upsilon_t^{\mu,\tilde\mu}|^\beta&=\beta|\Upsilon_t^{\mu,\tilde\mu}|^{\beta-2}\<\Upsilon_t^{\mu,\tilde\mu}, b(X_t^\mu,\mu_t)-b(X_t^{\tilde\mu},\tilde\mu_t)\>\,\d t\\
&\le L_1\beta|\Upsilon_t^{\mu,\tilde\mu}|^{\beta }+L_1\beta|\Upsilon_t^{\mu,\tilde\mu}|^{\beta-2}\mathbb W_\beta(\mu_t,\tilde\mu_t)^2.
\end{align*}
Obviously, the second term in the inequality  above cannot be dominated by the linear combination of $|\Upsilon_t^{\mu,\tilde\mu}|^{\beta}$ and $\mathbb W_\beta(\mu_t,\tilde\mu_t)^\beta$ when, in particular, the quantity $|\Upsilon_t^{\mu,\tilde\mu}|$ approaches to zero.

Additionally,
if \eqref{EQ} is replaced by the following stronger one:
\begin{equation*}
\begin{split}
& 2\langle b(x,\mu_1)-b(y,\mu_2),x-y\rangle+  \nu(|f(x,\mu_1,\cdot)-f(y,\mu_2,\cdot)|^2\I_U(\cdot)) \\
&\le L_1(|x-y|+\W_\be(\mu_1,\mu_2))|x-y|,\quad x,y\in\R^d, \mu_1,\mu_2\in\mathscr P_\beta(\R^d),
\end{split}
\end{equation*}
then the proof concerning the contraction of $\Phi$ will become much more straightforward by the aid of  an approximate argument and  It\^o's  formula. One can see the proof of Theorem \ref{theorem3} in the next subsection for details.
\item[{\rm(ii)}] For the case  $\beta\in (0,1)$,  the proof above no longer  works due to the definition of the
 $L^\beta$-Wasserstein distance $\W_\beta$. In particular, the contractivity of $\Phi$  under $\W_{\be,\ga}$ is unavailable even   for $\gamma>0$ large enough. Nonetheless, concerning the case $\beta\in (0,1)$,   it is possible to demonstrate   existence of the strong solution  via the Schauder fixed point theorem; see \cite[Proposition 1]{CT} for related details.

 \item[{\rm(iii)}] Under Assumptions $({\bf A}_1)$-$({\bf A}_{3})$, we can also derive that, for fixed $T>0$ and any $p\in[1,\beta)$ with $\beta\in(1,2]$,
 there exists a constant $C_T'>0$ such that
 \begin{align*}
 \E\Big(\sup_{0\le t\le T}  |X_t|^p\Big)   \le C_T'(1+\E|X_0|^\be).
\end{align*}
 This can be achieved via the stochastic Gronwall inequality; see the derivation of \eqref{EQ3} in Theorem \ref{theorem3} below for more details.
\end{itemize}

\end{remark}

\subsection{Proof of Theorem \ref{POC}}
In this part, we move to finish the proof of Theorem \ref{POC}.
\begin{proof}[Proof of Theorem $\ref{POC}$] Below, to shorten the notation, for all $t\in [0,T]$ and $1\le i\le n$, we set $Q_t^i:=X_t^i-X_t^{i,n}$
 and  $\tilde{\mu}^n_t:=\frac{1}{n}\sum_{j=1}^{n}\delta_{X_t^j}$ for $t\in[0,T]$. By invoking the triangle inequality and the basic inequality: $(a+b)^p\le 2^{p-1}(a^p+b^p)$ for $a,b\ge0, $ it follows that
for any $t\in[0,T],$
  \begin{equation}\label{RW}
  \begin{split}
\W^p_p(\mu_t,\bar{\mu}^n_t)
&\le 2^{p-1}\big(\W^p_p(\bar{\mu}_t^n,\tilde{\mu}^n_t)+\W^p_p(\mu_t,\tilde{\mu}^n_t)\big)\\
&\le 2^{p-1}\frac{1}{n}\sum_{j=1}^n|Q^j_t|^p+2^{p-1}\W^p_p(\mu_t,\tilde{\mu}^n_t),
\end{split}
\end{equation}
where in the second inequality we exploited the fact that
 $\W^p_p(\bar{\mu}_t^n,\tilde{\mu}^n_t)\le \frac{1}{n}\sum_{j=1}^{n}|Q^j_t|^p $
 since   $\frac{1}{n}\sum_{j=1}^n\delta_{(X_t^j,\bar{X}_t^{j,n})}$ is a coupling of $\bar{\mu}_t$ and $\tilde{\mu}^n_t$. Consequently, the assertion \eqref{wPOC} follows
 from Gronwall's inequality,
 provided that there exist constants $C_1(T), C_2(T)>0$ such that for any $t\in[0,T]$ and $1\le j\le n,$
 \begin{align}\label{EEW}
 \E|Q^j_t|^p\le C_1(T)\int_0^t\E\W^p_p(\mu_s,\tilde{\mu}^n_s)\,\d s
 \end{align}
 and
 \begin{align}\label{EWE}
 \E\W^p_p(\mu_t,\tilde{\mu}^n_t)\le C_2(T)\phi_{p,\beta,d}(n),
 \end{align}
 where $\phi_{p,\beta}(n,d)$ was defined as in \eqref{EE*}, and the number $C_2(T)$ depends on the initial moment $\E|X_0|^\beta.$

Set for $\vv>0$, $r\in \R$ and $ x\in\R^d$,
\begin{align}\label{EE-}
U_{\vv,r}(x):=(\vv+|x|^2)^{\frac{r}{2}}.
\end{align}   It is easy to see that for $\vv>0$, $r\in\R$ and $ x\in\R^d$,
\begin{align}\label{EE--}
\nn U_{\vv,r}(x)=rxU_{\vv,r-2}(x).
\end{align} Thus,
we obtain  from It\^o's formula and $({\bf A}_{1}')$ that there exist constants $c_1,c_2>0$ such that for any $t\in[0,T]$ and $\vv>0,$
\begin{equation}\label{WT}
\begin{split}
U_{\vv,p}(Q_t^i)&\le \vv^{\frac{p}{2}}+\hat{M}_t^i\\
&\quad+ c_1\int_0^t\Big(
 U_{\vv,p-2}(Q_s^i)|Q_s^i| \big(|Q_s^i|+\W_p(\mu_s,\bar{\mu}^n_s)\big) +|Q_s^i|^p+\W_p^p(\mu_s,\bar{\mu}^n_s)\Big)\,\d s\\
 &\le \vv^{\frac{p}{2}}+c_2\int_0^t\big(
 U_{\vv,p}(Q_s^i)+\frac{1}{n}\sum_{j=1}^n|Q^j_s|^p+\W^p_p(\mu_s,\tilde{\mu}^n_s)\big)\,\d s+\hat{M}_t^i ,
\end{split}
\end{equation}
where $(\hat{M}
_t^i)_{t\ge0}$ is a local martingale. In particular, the first inequality in \eqref{WT} follows
exactly  the  line to derive \eqref{R} below and we also
took advantage of $X_0^i=X_0^{i,n}$ herein,
and we
 made use of Young's inequality and \eqref{RW} in the second inequality.
Then, taking expectations on both sides of \eqref{WT} followed by sending $\vv\to0$,
the estimate \eqref{WT} enables us to deduce that for any $t\in[0,T],$
\begin{align*}
 \max_{1\le i\le N}\E|Q^i_t|^p
\le 2c_2\int_0^t\max_{1\le i\le N}\E|Q^i_s|^p\,\d t+c_2 \int_0^t\E\W^p_p(\mu_s,\tilde{\mu}^n_s)\,\d s.
\end{align*}
Whence, \eqref{EEW} follows from  Gronwall's  inequality.

In terms of \cite[Theorem 1]{FG}, for all $1\le p<\beta$, there exists a constant $c_3>0$ such that
\begin{align*}
 \E\W^p_p(\mu_t,\tilde{\mu}^n_t)\le c_3 \big(\E|X_t^i|^\beta\big)^{\frac{p}{\beta}}\phi_{p,\beta,d}(n),\quad t\in[0,T].
\end{align*}
As a result, \eqref{EWE} is available by taking \eqref{B0} into account.

Next, by applying the stochastic Gronwall inequality (see e.g. \cite[Lemma 3.7]{XZ}) and then approaching $\vv\to0$, we derive from \eqref{WT} that for any $0<q_1<q_2<1$,
\begin{align*}
\Big(\E\Big(\sup_{0\le s\le t}|Q_t^i|^{pq_1}\Big)\Big)^{\frac{1}{q_1}}\le c_2\Big( \frac{q_2}{q_2-q_1}\Big)^{\frac{1}{q_1}}\e^t \int_0^t\Big(
\frac{1}{n}\sum_{j=1}^n\E|Q^j_s|^p+\E\W^p_p(\mu_s,\tilde{\mu}^n_s)\Big)\,\d s.
\end{align*}
This, together with \eqref{EEW} and \eqref{EWE}, implies that there exists a constant $C_3(T)>0$ such that for any $0<q_1<q_2<1$ and $t\in[0,T],$
\begin{align*}
 \E\Big(\sup_{0\le s\le t}|Q_t^i|^{pq_1}\Big) \le  \frac{q_2}{q_2-q_1} \big(C_3(T)\phi_{p,\beta,d}(n)\big)^{q_1}.
\end{align*}
Thus, \eqref{sPOC} follows immediately.
\end{proof}

To proceed, we make a comment on  the method  proving   Theorem \ref{POC}.
\begin{remark}
\rm   By applying  It\^o's formula and BDG's inequality, along with \cite[Theorem 1]{FG},
we can also prove \eqref{sPOC} with $pq_1$ therein being replaced by $\beta$
as soon as the order of the initial moment is  greater than $\beta$. In this regard, the methods based respectively on the stochastic Gronwall inequality and BDG's inequality share the same feature.  Regarding the latter approach, one further needs to handle the term $\E(\sup_{0\le t\le T}M_t^i)$, where
\begin{align*}
M_t^i:=\int_0^t\int_U\big(|X^{i,n}_{s-}-X_{s-}^i+f(X^{i,n}_{s-}, z)-f(X^{i}_{s-}, z)|^p-|X^{i,n}_{s-}-X_{s-}^i|^p\big)\,\wt N^i(\d s,\d z).
\end{align*}
To this end, some additional assumptions  associated with $f$ (e.g., $f$ is Lipschitz in the state variable) have to be  exerted   provided that  mere \eqref{EQ4} is imposed.
\end{remark}

At the end of this part, we present the  proof of the  statement in Remark \ref{EX}(ii).

\begin{proof}[Proof of Remark $\ref{EX}${\rm(ii)}]
Below, we stipulate $x,y\in\R^d$ and $\mu_1,\mu_2\in\mathscr P_\beta(\R^d)$.
It is easy to see that
\begin{align*}
&2\<b(x,\mu_1)-b(y,\mu_2), x-y\>+  \nu(|f(x,\mu_1,\cdot)-f(y,\mu_2,\cdot)|^2\I_U(\cdot))\\
&\le 2\<x-y, C_1 (x-y)-C_2 (x|x|^2-y|y|^2)\>\\
&\quad+2d^{\frac{1}{2}}|x-y|\big|\mu_1(|h(x-\cdot)|^\beta)^{\frac{1}{\beta}} -\mu_2(|h(x-\cdot)|^\beta)^{\frac{1}{\beta}}  \big|\\
&\quad+C_3^2\nu(|\cdot|^2\I_U(\cdot))\big(C_4|x|^2-C_4|y|^2+ \mu_1(|h(x-\cdot)|^\beta)^{\frac{1}{\beta}} -\mu_2(|h(x-\cdot)|^\beta)^{\frac{1}{\beta}}  \big)^2
\end{align*}
and that for $z\in\R^d,$
\begin{align*}
|g(x,\mu_1,z)-g(y,\mu_2,z)| =|{\bf 1}+ z |\big||x|-|y|+ \mu_1(|h(x-\cdot)|^\beta)^{\frac{1}{\beta}} -\mu_2(|h(x-\cdot)|^\beta)^{\frac{1}{\beta}}  \big|.
\end{align*}

Via the Minkowski inequality, one obviously has
\begin{equation}\label{EW-}
\begin{split}
&\big| \mu_1(|h(x-\cdot)|^\beta)^{\frac{1}{\beta}} -\mu_2(|h(x-\cdot)|^\beta)^{\frac{1}{\beta}}  \big|\\
&=\Big|\Big(\int_{\R^d\times\R^d}|h(x-z_1)|^\beta\pi(\d z_1,\d z_2)\Big)^{\frac{1}{\beta}}-\Big(\int_{\R^d\times \R^d}|h(x-z_2)|^\beta\pi(\d z_1,\d z_2)\Big)^{\frac{1}{\beta}}\Big|\\
&\le \Big(\int_{\R^d\times\R^d}|h(x-z_1)-h(x-z_2)|^\beta\pi(\d z_1,\d z_2)\Big)^{\frac{1}{\beta}}\\
&\le \|h\|_{\rm Lip} \Big(\int_{\R^d\times\R^d}| z_1 - z_2 |^\beta\pi(\d z_1,\d z_2)\Big)^{\frac{1}{\beta}},
\end{split}
\end{equation}
where $\|h\|_{\rm Lip}$ means the Lipschitz constant of $h$, and $\pi\in\mathscr C(\mu_1,\mu_2)$.  Thus, taking infimum with respect to $\pi$ on both sides  of \eqref{EW-} yields that
\begin{align}\label{EW}
 \big| \mu_1(|h(x-\cdot)|^\beta)^{\frac{1}{\beta}} -\mu_2(|h(x-\cdot)|^\beta)^{\frac{1}{\beta}}  \big| \le \|h\|_{\rm Lip} \mathbb W_\beta(\mu_1,\mu_2).
\end{align}
Next, a direct calculation shows that
\begin{align*}
-\<x-y,x|x|^2-y|y|^2\>\le (1- (|x|^2+|y|^2)/6 )|x-y|^2,
\end{align*}
and it is easy to see that
\begin{align*}
(|x|^2-|y|^2)^2\le 2(|x|^2+|y|^2)|x-y|^2.
\end{align*}
Consequently, \eqref{EQ} follows from the basic inequality: $2ab\le a^2+b^2, a,b\in\R$, and making use of $\nu(|\cdot|^2\I_U(\cdot))<\8$
and
$C_2>12C_3^2C_4^2\nu(|\cdot|^2\I_U(\cdot))$. Apparently, \eqref{EQ*} is verifiable based on \eqref{EW}. Therefore, Assumption (${\bf A}_1$) is examinable. In terms of  definitions of $b$ and $f$ ,
it is easy to see that (${\bf A}_2$) is satisfied in case of $\nu(|\cdot|^2\I_U(\cdot))<\8$
and
$C_2>12C_3^2C_4^2\nu(|\cdot|^2\I_U(\cdot))$,  and that (${\bf A}_3$) holds true readily.    The proof is thus complete.
\end{proof}

\section{Extension to McKean-Vlasov SDEs with common noise}
In this section, we
consider the following McKean-Vlasov SDE with common noise:
\begin{equation} \label{common}
\begin{split}
\d X_t=&b(X_t,\scr{L}_{X_t|\scr{F}_t^{N^0}})\,\d t+\int_U f( X_{t-} ,z)
\,\wt N(\d t,\d z)\\
&+\int_V g(X_{t-},\scr{L}_{X_t|\scr{F}_t^{N^0}},z)
\,N(\d t,\d z)+\int_U f^0(X_{t-},z)
\,\wt N^0(\d t,\d z)\\
&+\int_V g^0(X_{t-},\scr{L}_{X_t|\scr{F}_t^{N^0}},z)
\,N^0(\d t,\d z),
\end{split}
\end{equation}
where $b:\R^{d}\times\scr{P}(\R^{d})\to\R^{d}$, $f,f^0:\R^d\times \R^d\to\R^d, $
and
$g,g^0:\R^{d}\times\scr{P}(\R^{d})\times \R^d\to\R^{d}$
are measurable maps; $U,V\subset \R^d\setminus \{{\bf0}\}$ so that $U\cap V=\emptyset$;
Poisson measures $N(\d t, \d z)$ and $N^0(\d t, \d z)$  correspond to the idiosyncratic noise and the common noise with L\'evy measure $\nu$ and $\nu^0$, respectively,  while $\wt N(\d t, \d z)$ and $\wt N^0(\d t, \d z)$ are their associated compensated Poisson measures.

To distinguish between the  underlying sources of randomness, we introduce complete probability spaces $(\Om^1,\scr F^1,\P^1)$ and $(\Om^0,\scr F^0,\P^0)$, whose respective filtrations $(\scr F^1_t)_{t\ge0}$ and $(\scr F^0_t)_{t\ge0}$ satisfy the usual conditions. In \eqref{common}, $N(\d t, \d z)$ and $N^0(\d t, \d z)$ shall be supported respectively on $\Om^1\times\R_+\times\R^d$ and $\Om^0\times\R_+\times\R^d$.
Throughout this section, we shall work on the product probability space $(\Om,\scr F,\F,\P)$, where $\Om:=\Om^0\times\Om^1$, $(\scr F,\P)$ is the completion of $(\scr F^0\otimes\scr F^1,\P^0\otimes\P^1)$ and $\F$ is the complete and right-continuous argumentation of $(\scr F^0_t\otimes\scr F^1_t)_{t\ge0}$.
Moreover, we write $\E^0$, $\E^1$ and $\E$ as the expectations on $(\Om^0,\scr F^0,\P^0)$, $(\Om^1,\scr F^1,\P^1)$ and $(\Om,\scr F,\F,\P)$,  respectively.
Note that $\scr{L}_{X_t|\scr{F}_t^{N^0}}$ represents the conditional distribution with respect to the $\si$-algebra $\scr{F}_t^{N^0}:=\si\{Z_s^0:s\le t\}$, in which  $$Z_t^0:=\int_0^t\int_U z \,\wt N^0(\d s,\d z)+\int_0^t\int_V z \,N^0(\d s,\d z).$$
Furthermore, in the subsequent analysis, we shall assume that the initial value $X_0$ is an $\scr F^1_0$-measurable random variable, so $(Z^0_t)_{t\ge0}$ is the solely common source of noise.

\subsection{Well-posedness of  McKean-Vlasov SDEs with common noise}

To carry out the study on the well-posedness of the SDE \eqref{common}, we impose the following assumptions.
\begin{enumerate}\it
  \item [{\rm(${\bf B}_0$)}] there is $\beta\in[1,2]$ so that
  \begin{align*}
 \nu(|f({\bf0},\cdot)|^2\I_V(\cdot))+\nu^0(|f^0({\bf0},\cdot)|^2\I_V(\cdot))&+\nu((1\vee|\cdot|^\beta \vee|g({\bf0,\delta_{\bf0}},\cdot)|^\beta\I_V(\cdot))\\
   &+\nu^0((1\vee|\cdot|^\beta \vee|g^0({\bf0,\delta_{\bf0}},\cdot)|^\beta\I_V(\cdot))<\8;
  \end{align*}
  \item [{\rm(${\bf B}_1$)}]
for  fixed $\mu\in\scr{P}_\be(\R^d)$ and $z\in\R^d$,   $\R^d \ni  x \mapsto b(x,\mu)$, $\R^d\ni x\mapsto f(x,\mu,z)$  and $\R^d\ni x\mapsto f^0(x,\mu,z)$ are continuous and locally bounded, and there exists a constant $K_1>0$ such that
for any   $x,y,z\in\R^d$,
 and $\mu_1,\mu_2\in\scr{P}_\be(\R^d)$,
\begin{equation}\label{EQ1}
\begin{split}
& 2\langle b(x,\mu_1)-b(y,\mu_2),x-y\rangle+  \nu(|f(x,\cdot)-f(y,\cdot)|^2\I_U(\cdot)) \\
& +\nu^0(|f^0(x,\cdot)-f^0(y,\cdot)|^2\I_U(\cdot)) \le K_1|x-y|(|x-y|+\W_\be(\mu_1,\mu_2))
\end{split}
\end{equation}
and
\begin{equation}\label{EQ2}
\begin{split}
&|g(x,\mu_1,z)-g(y,\mu_2,z)|+|g^0(x,\mu_1,z)-g^0(y,\mu_2,z)|\\
& \le K_1(1+|z|)(|x-y| +\W_\be(\mu_1,\mu_2) );
\end{split}
\end{equation}

  \item [{\rm(${\bf B}_2$)}] there exists a constant $K_2>0$ such that for any $x\in\R^d$ and  $\mu\in\scr{P}_\be(\R^d)$,
\begin{align*}
2&\<x,b(x,\mu)\>
 +\nu(|f(x,\cdot)|^2\I_U(\cdot))
 +\nu^0(|f^0(x,\cdot)|^2\I_U(\cdot))\\
 &\le K_2 \big(1+|x|^2+|x|\mu(|\cdot|^\be)^{\frac{1}{\beta} } \big);
\end{align*}

  \item [{\rm(${\bf B}_3$)}]for any $T,R>0$ and $\mu\in C([0,T];\mathscr P_\beta(\R^d))$,
\begin{align*}
\int_0^T \bigg(&\sup_{\{|x|\le R\}} |b(x,\mu_t)| +\int_U \sup_{\{|x|\le R\}}|f (x, z)|^2\nu(\d z)+\int_U \sup_{\{|x|\le R\}}|f^0 (x, z)|^2\nu^0(\d z)\bigg) \,\d t<\8.
\end{align*}
\end{enumerate}

The main result in this part is stated as follows.
\begin{theorem}\label{theorem3}
Assume that Assumptions $({\bf B}_0)$-$({\bf B}_{3})$  hold, and suppose further $X_0\in L^\beta(\Omega^1\to\R^d,\mathscr F^1_0,\P^1)$.
Then, the McKean-Vlasov  SDE with common noise \eqref{common} admits a unique strong solution $(X_t)_{t\ge0}$ satisfying that,
for any fixed $T>0$, there exists a constant $C_T>0$ such that
\begin{align}\label{B1}
\E|X_t|^\beta \le C_T(1+\E|X_0|^\be),\quad 0\le t\le T.
\end{align}
Furthermore,
if $\beta\in (1,2]$, then for all $p\in [1,\beta)$ and $T>0$,
 there exists a constant $C_T'>0$ such that
 \begin{align}\label{EQ3}
 \E\Big(\sup_{0\le t\le T}  |X_t|^p\Big)   \le C_T'(1+\E|X_0|^\be).
\end{align}
\end{theorem}

\begin{proof}  To begin with, we introduce some  notations. Let
$$L^{0,\beta}(\scr P_\beta(\R^d)):=\big\{\mu:\Omega^0\to\scr P_\beta(\R^d)\big|\, \E^0(\mu(|\cdot|^\beta))<\8\big\}.$$
Then $\big(L^{0,\be
}(\scr P_\beta(\R^d) ),\mathcal W_\be
\big)$
is a complete metric space (see e.g. \cite[Lemma 1.2]{KNRS}) endowed  with the metric:
$$
\mathcal W_\beta(\mu_1,\mu_2):=\big(\E^0\W_\beta^\beta(\mu_1,\mu_2)\big)^{\frac{1}{\beta}},\quad \mu_1,\mu_2\in L^{0,\beta}(\scr P_\beta(\R^d) ),
$$
so  $C\big([0,T];L^{0,\be
}(\scr P_\be
(\R^d))\big)$ is also a complete metric space for any fixed $T>0$.   In addition, we set for a fixed horizon $T>0,$
$$
\scr{D}^{X_0}_{T}=\Big\{\mu\in C([0,T];L^{0,\beta}(\scr{P}_\be(\R^{d}))): \mu_0=\scr{L}_{X_0},\sup_{t\in[0,T]}\E^0(\mu_t(|\cdot|^\beta))<\8\Big\},
$$
in which $X_0\in L^\be(\Omega^1\to\R^d,\scr{F}^1_0,\P^1)$ is the initial value of $(X_t)_{t\ge0}$, and
\begin{align*}
C([0,T];L^{0,\beta}(\scr{P}_\be(\R^{d}))):=\big\{\mu:[0,T]\times\Omega^0\to\scr P_\beta(\R^d) \mbox{ is weakly continuous}\big\}.
\end{align*}
 For $\eta>0,$ $(\scr{D}^{X_0}_{T}, \mathcal W_{\be,\eta})$ is a complete metric space equipped with the metric
$$
\mathcal W_{\be,\eta}(\mu,\widetilde{\mu}):=\sup_{0\le t\le T}\big(\e^{-\eta t}\mathcal W_\be(\mu_t,\widetilde{\mu}_t)\big), \quad \mu,\widetilde{\mu}\in\scr{D}^{X_0}_{T}.
$$

For $\mu\in \scr{D}^{X_0}_{T}$, we focus on the following SDE with random coefficients:
\begin{equation} \label{common1}
\begin{split}
X^\mu_t=&X_0+\int_0^t b(X^\mu_s,\mu_s)\,\d t+\int_0^t\int_U f(X^\mu_{s-}, z)\,\wt N(\d s,\d z)\\
&+\int_0^t\int_V g(X^\mu_{s-},\mu_s,z)\,N(\d s,\d z)+\int_0^t\int_U f^0(X^\mu_{s-}, z)\,\wt N^0(\d s,\d z)\\
&+\int_0^t\int_V g^0(X^\mu_{s-},\mu_s,z)\,N^0(\d s,\d z).
\end{split}
\end{equation}
Under $({\bf B}_0)$-$({\bf B}_{3})$, for each $\mu \ni \scr{D}^{X_0}_{T}$, \eqref{common1}  has a unique solution $(X_t^\mu)_{t\ge0}$ with the aid of Theorem \ref{thm} (which is  still available  to the SDE \eqref{e} with random coefficients). Accordingly, we
can  define a map $\scr{D}^{X_0}_{T}\ni\mu\mapsto\Ga(\mu)$ by
\begin{equation}\label{Ga}
(\Ga(\mu))_t=\scr{L}_{X^{\mu}_t|\scr F_t^{N^0}}, \quad t\in[0,T].
\end{equation}

By It\^o's formula, it follows from $({\bf B}_0)$, \eqref{EQ2} and $({\bf B}_2)$ that for some $C_1,C_2>0,$
\begin{equation}\label{EP1}
\begin{split}
\d(1+|X_t^\mu|^2)^{\frac{\beta}{2}}&\le \frac{K_2\beta}{2}(1+|X_t^\mu|^2)^{\frac{\beta}{2}-1}  \big(1+|X_t^\mu|^2+ |X_t^\mu|\mu_t(|\cdot|^\be)^{\frac{1}{\beta} }  \big)\,\d t\\
&\quad+C_1(1+|X_t^\mu|^\beta+\mu_t(|\cdot|^\beta))\,\d t+ \d\hat M_t^\mu
\\
&\le C_2\big((1+|X_t^\mu|^2)^{\frac{\beta}{2}}+\mu_t(|\cdot|^\beta)\big)\,\d t+ \d\hat M_t^\mu,
\end{split}
\end{equation}
where $(\hat M_t^\mu)_{t\ge0}$ is a martingale.  Then, via Gronwall's inequality, we have
\begin{align*}
\E|X_t^\mu|^\beta\le\Big(\E^1 (1+|X_0^\mu|^2)^{\frac{\beta}{2}}+C_2\int_0^T\E^0\mu_t(|\cdot|^\beta)\,\d t\Big)\e^{C_2T},\quad t\in[0,T].
\end{align*}
This, together with the fact that
\begin{align*}
\E^0((\Ga(\mu))_t(|\cdot|^\be))
 = \E^0\big(\E^1(|X_t^\mu|^\be|\scr F_t^{N^0})\big) = \E|X_t^\mu|^\be,
\end{align*}
implies that for $\mu\in\scr{D}^{X_0}_{T}$,
\begin{align*}
\sup_{0\le t\le T}\E^0\big((\Ga(\mu))_t(|\cdot|^\be)\big)<\8.
\end{align*}
 Next,  note that for any $h\in \mbox{Lip}_b(\R^d)$ and $t\in[0,T]$,
\begin{align*}
\big|\E^0\big((\Ga(\mu))_t(h(\cdot))\big)-\E^0\big((\Ga(\mu))_0(h(\cdot))\big)\big|=&
\E^0\big|\E^1\big((h(X_t^\mu ) -h(X_0^\mu))\big|\scr F_t^{N^0}\big)\big|\\
\le & \E|h(X_t^\mu ) -h(X_0^\mu)|.
\end{align*}
Thus,  we can conclude $\Ga(\mu)\in C([0,T];L^{0,\beta}(\scr{P}_\be(\R^{d})))$ by following the line to derive \eqref{WR} so that we arrive at
$\Ga(\mu)\in\scr{D}^{X_0}_{T}$.

In the sequel, we shall claim  that
$\Ga$ is contractive under $\mathcal W_{\be,\eta}$  for some appropriate $\eta>0$. According to \eqref{common1}, for $R^{\mu,\tilde{\mu}}_t:=X_t^\mu-X_t^{\tilde{\mu}}$ with  $\mu,\tilde{\mu}\in\scr{D}^{X_0}_{T}$,  we have
\begin{align*}
\d R^{\mu,\tilde{\mu}}_t=&B_t\,\d t+\int_U F_t(z)\,\wt N(\d t,\d z)+\int_V G_t(z)\,N(\d s,\d z)\\
&+\int_U F^0_t(z)\,\wt N^0(\d t,\d z)+\int_V G^0_t(z)\,N^0(\d s,\d z),
\end{align*}
where
$B_t:= b(X^\mu_t,\mu_t)- b(X^{\tilde\mu}_t,\tilde{\mu}_t)$ and
$$F_t(z):=f(X_{t-}^\mu, z)-f(X_{t-}^{\tilde{\mu}}, z), \quad F^0_t(z):=f^0(X_{t-}^\mu, z)-f^0(X_{t-}^{\tilde{\mu}}, z),$$
$$G_t(z):=g(X_{t-}^\mu,\mu_t,z)-g(X_{t-}^{\tilde{\mu}},\tilde{\mu}_t,z), \quad G^0_t(z):=g^0(X_{t-}^\mu,\mu_t,z)-g^0(X_{t-}^{\tilde{\mu}},\tilde{\mu}_t,z).$$

Recall that $U_{\vv,\be}$ is defined as in \eqref{EE-}. Then, applying It\^o's formula  and making use of  $({\bf B}_{1})$ and \eqref{EE--}   yield that
\begin{equation}\label{EY4}
\begin{split}
\d U_{\vv,\be}(R^{\mu,\tilde{\mu}}_t)
&\le \frac{\be}{2}U_{\vv,\beta-2}(R^{\mu,\tilde{\mu}}_t)\big(2\la R^{\mu,\tilde{\mu}}_t, B_t\ra+\nu(|F_t(\cdot)|^2\I_U(\cdot))+\nu^0(|F^0_t(\cdot)|^2\I_U(\cdot))\big)\,\d t\\
&\quad +2^{\frac{\beta}{2}} \big(\nu(|G_t(\cdot)|^\be\I_V(\cdot))+\nu^0(|G^0_t(\cdot)|^\be\I_V(\cdot))\big)\,\d t\\
&\quad +|R^{\mu,\tilde{\mu}}_t|^\be(\nu(\I_V)+\nu^0(\I_V))\,\d t+\d \hat{M}_t\\
&\le \frac{\be K_1}{2}U_{\vv,\beta-1}(R^{\mu,\tilde{\mu}}_t)(|R^{\mu,\tilde{\mu}}_t|+\W_\be(\mu_t,\tilde{\mu}_t))\,\d t\\
&\quad +c_1\,(|R^{\mu,\tilde{\mu}}_t|^\be+\W^\be_\be(\mu_t,\tilde{\mu}_t))\,\d t+|R^{\mu,\tilde{\mu}}_t|^\be(\nu(\I_V)+\nu^0(\I_V))\,\d t+\d \hat{M}_t,
\end{split}
\end{equation}
where  $c_1:=2^{\frac{\beta}{2}} K_1^\be(\nu((1+|\cdot|)^\beta\I_V(\cdot))+\nu^0((1+|\cdot|)^\beta\I_V(\cdot)))<\8$ thanks to $({\bf B}_0)$. Whereafter, integrating from $0$ to $t$ followed by taking expectations on both sides of \eqref{EY4}, and applying Young's inequality and
the fact that $X^\mu_t=X^{\tilde \mu}_t=X_0$, we obtain there exists a constant $C_T^*>0$ that for any $t\in[0,T]$,
\begin{equation}\label{R}
\begin{split}
\E U_{\vv,\be}(R^{\mu,\tilde{\mu}}_t)
&\le C_T^*\int_0^t\big(\E  U_{\vv,\be}(R^{\mu,\tilde{\mu}}_s)+\E^0\W^\be_\be(\mu_s,\tilde{\mu}_s)\big)\,\d s.
\end{split}
\end{equation}
This, combining with Gronwall's inequality and approaching $\vv\rightarrow0$, leads to
$$\E|R^{\mu,\tilde{\mu}}_t|^{\be}\le C_T^* e^{C_T^* T}\int_0^t\E^0\W_{\be}^{\be}(\mu_s,\tt\mu_s)\,\d s,\quad t\in[0,T].$$
Correspondingly, we derive that
\begin{align*}
\mathcal  W_{\be,\eta}^\be(\Ga(\mu),\Ga(\tt{\mu}))
&\le\sup_{0\le t\le T}\big(\e^{-\eta\be t}\E^0\big(\E^1\big(|R^{\mu,\tilde{\mu}}_t|^{\be}\big|\scr F_t^{N^0}\big)\big)\big)\\
&\le C_T^* e^{C_T^* T}\sup_{0\le t\le T}\big( \int^t_0\e^{-\eta\be(t-s)}\e^{-\eta\be s}\,\E^0\W_\be^\be(\mu_s,\tt{\mu}_s)\,\d s\big)\\
&\le  C_T^* e^{C_T^* T}/(\eta\be)\mathcal W_{\be,\eta}^\be(\mu,\tt{\mu}).
\end{align*}
As a consequence, we conclude that
 $\Ga$ is contractive under $\mathcal W_{\be,\eta}$  for $\eta>0$ large enough so the strong well-posedness of \eqref{common} is available
 via the Banach fixed point theorem.

 The assertion \eqref{B1} follows by following the procedure to derive \eqref{EP1} and applying Gronwall's inequality. Next, by virtue of the stochastic Gronwall inequality (see e.g. \cite[Lemma 3.7]{XZ}),  we obtain from \eqref{EP1} that for any $0<q_1<q_2<1$ and $\mu\in\scr{D}^{X_0}_{T}$,
\begin{align*}
\bigg(\E\Big(\sup_{0\le t\le T}(1+|X_t^\mu|^2)^{\frac{q_1\beta}{2}}\Big)\bigg)^{\frac{1}{q_1}}\le \Big( \frac{q_2}{q_2-q_1}\Big)^{\frac{1}{q_1}}\e^T\bigg(\E^1(1+|X_0^\mu|^2)^{\frac{\beta}{2}}+C_2\int_0^T\E^0\mu_t(|\cdot|^\beta)\,\d t\bigg).
\end{align*}
In particular, we take  $\mu\in\scr{D}^{X_0}_{T}$ as the fixed point of $\Gamma(\mu)$, defined in \eqref{Ga}, such that $X_t^\mu=X_t$ for any $t\in[0,T]$ and
\begin{align*}
\bigg(\E\Big(\sup_{0\le t\le T}(1+|X_t|^2)^{\frac{q_1\beta}{2}}\Big)\bigg)^{\frac{1}{q_1}}\le \Big( \frac{q_2}{q_2-q_1}\Big)^{\frac{1}{q_1}}\e^T\bigg(\E^1(1+|X_0|^2)^{\frac{\beta}{2}}+C_2\int_0^T\E|X_t|^\beta\,\d t\bigg).
\end{align*}
As a result, \eqref{EQ3} holds true from \eqref{B1}.
\end{proof}

At the end of this subsection, we make a remark concerning Assumptions (${\bf B}_1$) and (${\bf B}_2$).

\begin{remark}\rm
As far as the decoupled SDE associated with the  McKean-Vlasov SDE \eqref{MV} is concerned, the frozen
measure variable  is deterministic  so the interlacing technique is applicable and moreover  the corresponding technical condition is weaker;
see $({\bf A}_1)$ and $({\bf A}_2)$ for more details. Whereas, with regard to the SDE with random coefficients corresponding to the conditional McKean-Vlasov SDE \eqref{common}, the underlying measure-valued process is no longer deterministic but random. Thus, the interlacing trick adopt in the proof of Theorem \ref{theorem2} is unusable.  Furthermore,  once we replace the term $|X_t^\mu|\mu_t(|\cdot|^\be)^{\frac{1}{\beta} }$  in \eqref{EP1} by $\mu_t(|\cdot|^\be)^{\frac{2}{\beta} }$, we need to estimate correspondingly the quantity  $\E((1+|X_t^\mu|^2)^{\frac{\beta}{2}-1}\mu_t(|\cdot|^\be)^{\frac{2}{\beta} })$. In case that $\mu_t$ is deterministic, it is easy to bound the term mentioned. Nevertheless,
$\E^0(\mu_t(|\cdot|^\be)^{\frac{2}{\beta} })$ might explode for $\mu\in \scr{D}^{X_0}_{T} $ $($in this case,    $(\mu_t)_{t\ge0}$ is a measure-valued stochastic process$)$. On the basis of the aforementioned analysis, we impose Assumptions $({\bf B}_1)$ and $({\bf B}_2)$, which is a little bit stronger than Assumptions $({\bf A}_1)$ and $({\bf A}_2)$,
to offset the singularity arising from the spatial variables.
\end{remark}

\subsection{Conditional propagation of chaos for McKean-Vlasov SDEs with common noise} In this subsection, we are still concerned with the
 L\'evy-driven McKean-Vlasov SDE  with common noise \eqref{common}, which  describes the asymptotic behavior of the mean-field interacting particle system below:
\begin{equation}\left\{\begin{array}{l} \label{commonN}
\d \bar{X}^{i,n}_t=b(\bar{X}^{i,n}_t,\bar{\mu}^n_t)\,\d t+\displaystyle\int_U f(\bar{X}^{i,n}_{t-},z)
\,\wt N^i(\d t,\d z)
+\int_V g(\bar{X}^{i,n}_{t-},\bar{\mu}^n_{t-},z)
\,N^i(\d t,\d z)\\
\qquad\qquad  +\displaystyle\int_U f^0(\bar{X}^{i,n}_{t-},z)
\,\wt N^{0,i}(\d t,\d z)
+\int_V g^0(\bar{X}^{i,n}_{t-},\bar{\mu}^n_{t-},z)
\,N^{0,i}(\d t,\d z),\\
~\bar{X}^{i,n}_0=X_0^i, \qquad i=1,2,\cdots,n,\end{array}\right.
\end{equation}
where $\bar{\mu}^n_t:=\frac{1}{n}\sum_{i=1}^n\delta_{\bar{X}^{i,n}_{t}}$,    $\bar{\mu}^n_{t-}:=\frac{1}{n}\sum_{i=1}^n\delta_{\bar{X}^{i,n}_{t-}},
$
 and $\{N^i(\d t, \d z)\}_{1\le i\le n}$ (resp. $\{N^{0,i}(\d t, \d z)\}_{1\le i\le n}$) are independent Poisson measures with intensity measure $\d t\times \nu(\d z)$ (resp. $\d t\times \nu^0(\d z)$).

Throughout this subsection, we will assume that $\beta\in (1,2]$ and work under Assumptions $({\bf B}_0)$-$({\bf B}_{3})$ with $\be$ involved in Assumption $({\bf B}_1)$ replaced by $p\in[1,\be)$. It is easy to see that \eqref{commonN} has a unique strong solution $(\bar{X}^{i,n}_t)_{t\ge0}$. Denote by $\{(X_t^i)_{t\ge0}\}_{1\le i\le n}$ $n$-independent versions of the unique solution to \eqref{common}. In particular, $(\mu_t)_{t\ge0}$ is their common distribution,

It is worth noting that in the presence of common noise, all particles in the stochastic system \eqref{commonN} are not asymptotically independent any more
and the classical  propagation of chaos no longer holds. Whereas, \cite[Theorem 2.12]{CD} puts forward the conditional propagation of chaos, which reveals that, conditioned on the $\si$-algebra associated with common noise,
all particles are asymptotically independent and the empirical measure converges to the common conditional distribution of each particle.
The specific result upon conditional propagation of chaos in our setting is as follows.

\begin{theorem}\label{NPOC}
Assume that $\beta\in (1,2]$, that Assumptions $({\bf B}_0)$-$({\bf B}_{3})$ hold  with $\be$ involved in Assumption $({\bf B}_1)$ replaced by some $p\in[1,\be)$, and suppose further $X^i_0\in L^\beta(\Omega^1\to\R^d,\mathscr F^1_0,\P^1)$ for any $1\le i\le n$.
Then, for any fixed $T>0$, there exists a constant $\bar C_T>0$ such that
\begin{equation*}\label{wNPOC}
 \E\W_p^p(\bar{\mu}^n_t,\mu_t)\le \bar C_T\phi_{p,\beta,d}(n),\quad t\in[0,T],
\end{equation*}
where $\phi_{p,\beta}(n,d)$ was defined as in \eqref{EE*}. Furthermore,
 for fixed $T>0$
and
any  $0\le q_1<q_2<1,$
there exists a constant $\hat C_T>0$ such that
\begin{equation}
\begin{split}
\E\Big(\sup_{t\in[0,T]}|\bar{X}^{i,n}_t-X_t^i|^p\Big)\le \frac{q_2}{q_2-q_1}\big(\hat C_T
\phi_{p,\beta,d}(n)\big)^{q_1}.
\end{split}
\end{equation}

\end{theorem}

\begin{proof}
The structure of proof is largely analogous to Theorem \ref{POC}, so we omit it here.
\end{proof}

\vspace{0.3cm}
\noindent \textbf{Acknowledgements.} The research of Jianhai Bao is supported by the National Key R\&D Program of China (2022YFA1006004) and  the National Natural Science Foundation of China (Nos.\ 12071340).
The research of Jian Wang is supported by the National Key R\&D Program of China (2022YFA1006003) and  the National Natural Science Foundation of China (Nos.\ 12071076 and 12225104).

\end{document}